\documentclass[review]{elsarticle}

\usepackage{bm,amssymb,amsmath,mathrsfs}
\usepackage{amsthm}
\usepackage{lineno}
\usepackage[usenames]{color}

\usepackage{dcolumn}

\theoremstyle{plain}
\newtheorem{theorem}{Theorem}[section]
\newtheorem{remark}[theorem]{Remark}
\newtheorem{conjecture}[theorem]{Conjecture}
\newtheorem{proposition}[theorem]{Proposition}

\begin{document}

\journal{(internal report CC25-12)}

\begin{frontmatter}

\title{Identically vanishing $k$-generalized Fibonacci polynomials}

\author[cc]{S.~R.~Mane}
\ead{srmane001@gmail.com}
\address[cc]{Convergent Computing Inc., P.~O.~Box 561, Shoreham, NY 11786, USA}

\begin{abstract}
The recurrence for the $k$-Fibonacci polynomials
is usually iterated upwards to positive values of $n$ only.
When the recurrence is iterated downwards to $n<0$,
there are indices where the polynomials vanish identically.
This fact does not seem to have been noted in the literature.
We derive the set of such indices.
We establish the connection of our results to the solution of the Skolem problem for the $k$-Fibonacci numbers.
For $k\ge3$ and $n<0$, we show that the degree of the polynomial does not increase monotonically with $|n|$.
The so-called `left-justified $k$-nomial triangle' is extended to treat negative indices.
We derive expressions for the individual polynomial coefficients (the elementary symmetric polynomials of the roots).
We present results for the properties of the polynomials, for both $n>0$ and $n<0$,
including factorization of the polynomials and properties of the roots.
Results are also derived for real roots.
We present new, tighter, bounds on the amplitudes of the nonzero roots.
We derive new combinatorial sums for the polynomial coefficients,
which are more concise and computationally efficient than previously published expressions.
\end{abstract}


\end{frontmatter}

\newpage
\setcounter{equation}{0}
\section{Introduction}\label{sec:intro}
The so-called \textit{$k$-Fibonacci numbers} $F_{n,k}$ are a sequence of integers which satisfy the following recurrence:
\begin{equation}  
\label{eq:rec_Fnk_num}
F_{n,k} = F_{n-1,k} +\dots +F_{n-k,k} \,.
\end{equation}
By convention, the initial values are $F_{n,k}=0$ for $n\in[0,k-2]$ and $F_{n,k}=1$ for $n=k-1$.
The case $k=2$ gives the usual Fibonacci numbers (denoted by $F_n$ below).
The \textit{$k$-Fibonacci polynomials} are a generalization of the $k$-Fibonacci numbers.
They are a family of polynomials $\mathcal{F}_{n,k}(x)$, where $k\ge2$, which satisfy the following recurrence:
\begin{equation}  
\label{eq:rec_Fnk_pos}
\mathcal{F}_{n,k}(x) = x^{k-1}\mathcal{F}_{n-1,k}(x) +x^{k-2}\mathcal{F}_{n-2,k}(x) +\dots +\mathcal{F}_{n-k,k}(x) \,.
\end{equation}
By convention, the initial values are $\mathcal{F}_{n,k}(x)=1$ for $n=1$ and $\mathcal{F}_{n,k}(x)=0$ for $n\in[-(k-2),0]$.
The definition in eq.~\eqref{eq:rec_Fnk_pos} follows Hoggatt and Bicknell \cite{HoggattBicknell_FibGen} (who called them ``R-Bonacci'' numbers).
There are other generalizations of the Fibonacci polynomials, see e.g.~\cite{Ricci,Soykan}.
Arolkar defined the $B$-Tribonacci polynomials via the recurrence (\cite{Arolkar}, eq.~(1))
$t_{n+2} = x^2t_{n+1} +2xt_n + t_{n-1}$.
We do not treat such alternative definitions here.
Most authors iterate the recurrence eq.~\eqref{eq:rec_Fnk_pos}
upwards to positive values of $n$, see e.g.~\cite{Arslan_Uslu,Gupta_etal,HoggattBicknell_FibGen,Koshy,Tasyurdu_Polat}.
All the coefficients in the recurrence are positive, hence $\mathcal{F}_{n,k}(x)$ is not identically zero for $n\ge1$.
It is then easily seen that for $n\ge1$, $\mathcal{F}_{n,k}(x)$ is a polynomial in $x$ of degree $d_{n,k} = (k-1)(n-1)$.
We begin with $\mathcal{F}_1(x)=1$, which is a polynomial of degree zero,
and each upward step in $n$ adds $k-1$ to the degree of the polynomial
because $\mathcal{F}_{n,k}(x)=x^{k-1}\mathcal{F}_{n-1,k}(x)+\dots$.
For example $\mathcal{F}_{2,k}(x)=x^{k-1}$ and $\mathcal{F}_{3,k}(x)=x^{2(k-1)} +\dots$.

However, little attention has been paid if we iterate eq.~\eqref{eq:rec_Fnk_pos} \emph{downwards} to negative values of $n$.
Kuhapatanakul and Sukruan \cite{Kuhapatanakul_Sukruan} studied the Tribonacci numbers for negative indices
and Waddill \cite{Waddill} studied the Tetranacci numbers for both positive and negative indices.
For negative $n$, we rewrite the recurrence in eq.~\eqref{eq:rec_Fnk_pos} as follows:
\begin{equation}  
\label{eq:rec_Fnk_neg}
\mathcal{F}_{n,k}(x) = \mathcal{F}_{n+k,k}(x) -\sum_{j=1}^{k-1} x^j\mathcal{F}_{n+j,k}(x) \,.
\end{equation}
We shall show that, for certain negative indices $n$, the $k$-Fibonacci polynomials \emph{vanish identically}.
This fact does not seem to have been noted in the literature.
There are $k(k-1)/2$ such indices (including the block of $k-1$ initial values).
For $k\ge3$ and $n<0$, we shall also show the degree of the polynomial does \emph{not} increase monotonically with $|n|$.

It will be useful below to employ two parameters, a quotient $q$ and a remainder $r$.
\begin{subequations}
\label{eq:q_r_def}
\begin{align}
\label{eq:q_def}
q_{n,k} &= \begin{cases} \lfloor(n-2)/k\rfloor &\qquad\qquad\qquad (n>0) \,, \\
  \lfloor(|n|+1)/k\rfloor &\qquad\qquad\qquad (n\le0) \,. \end{cases}
\\
\label{eq:r_def}
r_{n,k} &= \begin{cases} (k-1)(n-1)\!\!\!\mod k &\qquad (n>0) \,, \\
  (|n|+1)\!\!\!\mod k &\qquad (n\le0) \,. \end{cases}
\end{align}
\end{subequations}
We shall omit the subscripts $n$ and $k$ unless necessary.
The expression for $r$ for $n>0$ is also valid for $n\le0$, but we avoid negative arguments.
The definition of $q$ is peculiar but not a misprint.
The evidence suggests that the real partition is not ``positive and negative $n$'' but rather $n>\frac12$ and $n<\frac12$,
so $q_{n,k}=\lfloor(n-\frac32)/k\rfloor=\lfloor(2n-3)/(2k)\rfloor$ for $n\ge1$
and $q_{n,k}=\lfloor(|n|+\frac32)/k\rfloor=\lfloor(2|n|+3)/(2k)\rfloor$ for $n\le0$,
which is equivalent to eq.~\eqref{eq:q_def}.
By construction, $kq+r = |n|+1 = 1-n$ for $n\le0$.
For $n>0$, the expression is $k(q+1)-r = n-1$.
Also $q\ge0$ and $r\in[0,k-1]$ for all $n\in\mathbb{Z}$.

The structure of this paper is as follows.
Sec.~\ref{sec:poly_prop} lists basic properties of the $k$-Fibonacci polynomials.
For $n>0$, Hoggatt and Bicknell \cite{HoggattBicknell_FibGen} showed how to obtain the polynomial coefficients
from the so-called \textit{left-justified $k$-nomial triangle}, to be explained below.
We extend the formalism to treat $n<0$ also.
We present more concise computationally efficient sums for the polynomial coefficients,
which do not seem to have been published before.
Sec.~\ref{sec:poly_neg} focuses on the polynomials for $n<0$.
In particular, we quantify the set of the identically vanishing polynomials.
Sec.~\ref{sec:Skolem} establishes the relation of our results to the solution of the Skolem problem for the $k$-Fibonacci numbers.
Sec.~\ref{sec:poly_fac} presents results for the factorization of the polynomials.
The nonzero roots have a $k$-fold rotational symmetry around the origin in the complex plane
(i.e.~if $x$ is a root, then $\omega_k x$ is also a root, where $\omega_k$ is a primitive $k^{th}$ root of unity).
Sec.~\ref{sec:poly_roots} presents results for properties of the polynomial roots,
including bounds on the amplitudes of the nonzero roots and results for real roots.
Sec.~\ref{sec:elem_symm_poly} presents expressions for the elementary symmetric polynomials of the nonzero roots,
in particular the sum and product of the nonzero roots.

\setcounter{equation}{0}
\section{Selected properties of polynomials}\label{sec:poly_prop}
\subsection{Examples of identically vanishing polynomials}\label{sec:zero_poly}
For $k=2$, the recurrence for the Fibonacci polynomials $F_n(x)$ is as follows:
\begin{equation}  
F_n(x) = xF_{n-1}(x) +F_{n-2}(x) \,.
\end{equation}
To iterate downwards to negative values of $n$, express this as follows:
$F_n(x) = -xF_{n+1}(x) +F_{n+2}(x)$.
If we replace $x$ by $-x$, we obtain basically the same recurrence as the original.
Given the initial conditions $F_0(x)=0$ and $F_1(x)=1$, we obtain the well-known result
$F_{-n}(x) = F_n(-x) = (-1)^{n-1}F_n(x)$.
However, this is \emph{not} the case for $k\ge3$.
From \cite{HoggattBicknell_FibGen}, the recurrence for the Tribonacci polynomials $T_n(x)$ is
\begin{equation}  
T_n(x) = x^2T_{n-1}(x) +xT_{n-2}(x) +T_{n-3}(x) \,.
\end{equation}
Flipping the sign of $x$ is not helpful because the coefficient of $x^2$ does not change sign.
The initial values are $T_1(x) = 1$ and $T_0(x)=T_{-1}(x) = 0$.
Then we obtain
$T_{-2}(x) = T_1(x) =1$,
$T_{-3}(x) = -xT_{-2}(x) =-x$ and
$T_{-4}(x) = -xT_{-3}(x) -x^2T_{-2}(x) =x^2 -x^2 =0$.
From \cite{HoggattBicknell_FibGen}, the recurrence for the Quadranacci polynomials $Q_n(x)$ is
\begin{equation}
Q_n(x) = x^3Q_{n-1}(x) +x^2Q_{n-2}(x) +xQ_{n-3}(x) +Q_{n-4}(x) \,.
\end{equation}
The initial values are $Q_1(x) = 1$ and $Q_0(x)=Q_{-1}(x)=Q_{-2}(x) = 0$.
Then we obtain
$Q_{-3}(x) = Q_1(x) =1$,
$Q_{-4}(x) = -xQ_{-3}(x) =-x$ and
$Q_{-5}(x) = -xQ_{-4}(x) -x^2Q_{-3}(x) =x^2 -x^2 =0$.
The pattern is the same for all the higher $k$-Fibonacci polynomials:
iterating downwards from $n=-(k-1)$, the first two are respectively $1$ and $-x$ and the third is identically zero.
There are also other negative indices where the polynomials vanish identically.
Table \ref{tb:Fib_poly_neg} tabulates the Tribonacci, Quadranacci and Pentanacci polynomials for $0 \ge n \ge -23$.
Blanks in the table indicate indices where the polynomials vanish identically.
There are $3$, $6$, and $10$ vanishing polynomials for $k=3$, $4$, and $5$ respectively.
The degree $d_{n,k}$ of the nonvanishing polynomials is also tabulated in Table \ref{tb:Fib_poly_neg}.
Observe that the degree of the polynomials does not increase monotonically with $|n|$.

\subsection{Left-justified $k$-nomial table of polynomial coefficients}\label{sec:table_poly_coeff}
Hoggatt and Bicknell \cite{HoggattBicknell_FibGen} showed how to obtain the
coefficients of the polynomial $\mathcal{F}_{n,k}(x)$
from the elements in the so-called \textit{left-justified $k$-nomial triangle}.
They treated the case $n>0$ only.
See \cite{HoggattBicknell_FibGen} for details.
We explain the concept and extend the formalism to treat $n<0$ also.
To avoid confusion of notation, we employ $n$ and $k$ to index a $k$-Fibonacci polynomial $\mathcal{F}_{n,k}(x)$.
We index the rows and columns in the tables below using $m$ and $j$, respectively.
We expand $(1+x+\dots+x^{k-1})^m$ in a power series in $x$.
The coefficients of the powers of $x$ are arranged in a table, where the row is indexed by $m\in\mathbb{Z}$ and the column $j\ge0$ is zero-indexed.
For $m\ge0$, the table has the shape of a triangle,
which is termed the \textit{left-justified $k$-nomial triangle}.
For $k=2$, it is the left-justified Pascal triangle.
The $m<0$, power series is non-terminating, but that does not matter.
We shall therefore employ the term \textit{left-justified $k$-nomial table} below.

We exhibit the $k$-nomial table for $k=4$, to obtain the Quadranacci polynomials.
We display the numbers in Table \ref{tb:tbl_k4}.
All elements to the left of the table are understood to be zero.
Elements tagged with the same superscript in Table \ref{tb:tbl_k4} are coefficients in the same Quadranacci polynomial.
We list a few Quadranacci polynomials below to demonstrate, for both $n\ge0$ and $n<0$.
We read the table along a diagonal from southeast to northwest.
At each step, we move up a row and shift left by $k$ columns, until we exit the table on the left.
\begin{equation}
\label{eq:Quad_from_table}
\begin{array}{llllllllllll}
  E: & Q_{5}(x) &=& 1\,x^{12} +3\,x^8 +3\,x^4 +1\,x^0  &=& x^{12} +3x^8 +3x^4 +1 \,, \\
  D: & Q_{4}(x) &=& 1\,x^9 +2\,x^5 +1\,x^1  &=& x^9 +2x^5 +x \,, \\
  C: & Q_{3}(x) &=& 1\,x^6 +1\,x^2 &=& x^6 +x^2 \,, \\
  B: & Q_{2}(x) &=& 1\,x^3 &=& x^3 \,, \\
  A: & Q_{1}(x) &=& 1\,x^0 &=& 1 \,, \\
  X: & Q_{0}(x) &=& 0 &=& 0 \,, \\
  Y: & Q_{-1}(x) &=& 0 &=& 0 \,, \\
  Z: & Q_{-2}(x) &=& 0 &=& 0 \,, \\
  a: & Q_{-3}(x) &=& 1\,x^0 &=& 1 \,, \\
  b: & Q_{-4}(x) &=& -1\,x^1 &=& -x \,, \\
  c: & Q_{-5}(x) &=& 0\,x^2 &=& 0 \,, \\
  d: & Q_{-6}(x) &=& 0\,x^3 &=& 0 \,, \\
  e: & Q_{-7}(x) &=& 1\,x^4+1\,x^0 &=& x^4+1 \,, \\
  f: & Q_{-8}(x) &=& -1\,x^5-2\,x &=& -x^5-2x \,, \\
  g: & Q_{-9}(x) &=& 0\,x^6+1\,x^2 &=& x^2 \,, \\
  h: & Q_{-10}(x) &=& 0\,x^7+0\,x^3 &=& 0 \,, \\
  i: & Q_{-11}(x) &=& 1\,x^8+2\,x^4+1\,x^0 &=& x^8+2x^4+1 \,, \\
  j: & Q_{-12}(x) &=& -1\,x^9-4\,x^5-3\,x^1 &=& -x^9-4x^5-3x \,, \\
  k: & Q_{-13}(x) &=& 0\,x^{10}+2\,x^6+3\,x^2 &=& 2x^6+3x^2 \,, \\
  \ell: & Q_{-14}(x) &=& 0\,x^{11}+0\,x^7-1\,x^3 &=& -x^3 \,.
\end{array}
\end{equation}
For $n>0$, it is easily verified that the above expressions are correct.
The entries tagged `$X$' through `$Z$' index the initial conditions $n\in\{0,-1,-2\}$ and the polynomials all vanish: $Q_{0}(x)=Q_{-1}(x)=Q_{-2}(x)=0$.
It is only when we reach indices where $n \le -(k-1)$, i.e.~$n\le-3$ in this case, that we again encounter nonvanishing polynomials.  
For $n\le-(k-1)$, the expressions above agree with the display in Table \ref{tb:Fib_poly_neg}.
Observe how $Q_{-5}(x)$, $Q_{-6}(x)$ and $Q_{-10}(x)$ vanish identically,
and the coefficients of the highest powers in $Q_{-9}(x)$, $Q_{-13}(x)$
and the two highest powers in $Q_{-14}(x)$ are zero,
hence the polynomials have lower degrees.
\begin{remark}
For $n>0$, Hoggatt and Bicknell \cite{HoggattBicknell_FibGen} read the left-justified $k$-nomial table along a diagonal at a $45^\circ$ angle as follows.
Start at row $m=n-1$ and column $j=0$ and at each step, move up a row and shift right by one column.
Such a procedure works only for $n>0$.
Our procedure (read a diagonal from southeast to northwest) works for all $n\in\mathbb{Z}$.
\end{remark}
Note the following, for all $k\ge2$:
\begin{enumerate}
\item
We foliate the left-justified $k$-nomial table with parallel diagonals.
The elements along a diagonal yield the coefficients of the corresponding $k$-Fibonacci polynomial.  
For the polynomial $\mathcal{F}_{n,k}(x)$, the equation for the diagonal is $j = km + 1-n$; it intersects the row indexed by $m=0$ at the column $j=1-n$.
\item
For fixed $k\ge2$ and $n>0$, the coefficient of the highest power of $x$ in $\mathcal{F}_{n,k}(x)$ is encountered at the row indexed by $m=n-1$
and is the rightmost nonzero element in that row, i.e.~the column index is $j=(k-1)(n-1)$.
\item
  The cases $n\in[-(k-2),0]$ are the initial values $\mathcal{F}_{n,k}(x)=0$ and the corresponding table elements are zero.
\item
For fixed $k\ge2$ and $n\le-(k-1)$,
the reading of the diagonal for $\mathcal{F}_{n,k}(x)$ begins at the row indexed by $m=-1$ and the column indexed by $j=|n|+1-k$.
\emph{However, the coefficient at this location may be zero, hence the degree of the polynomial may be lower, and the polynomial may vanish identically.}
\end{enumerate}

\subsection{Combinatorial sums for $k$-nomial table elements}\label{sec:tbl_elem}
Let $C_k(m,j)$ denote the coefficient of $x^j$ in the series expansion of $(1+x+\dots+x^{k-1})^m$ in powers of $x$.
It is computationally inefficient to expand $(1+x+\dots+x^{k-1})^m$ in a multinomial sum, to extract the coefficient of $x^j$.
We propose the following more concise expressions, which do not seem to have been published in the literature.
\begin{proposition}
For fixed $k\ge2$, $m\in\mathbb{Z}$ and $j\ge0$,
let $C_k(m,j)$ denote the coefficient of $x^j$ in the series expansion of $(1+x+\dots+x^{k-1})^m$.
For $m\ge0$, we obtain
\begin{equation}
\label{eq:C_pos}
  C_k(m,j) = \sum_{s=0}^{\lfloor j/k \rfloor} (-1)^s\binom{m}{s}\binom{m+j-ks-1}{m-1} \,.
\end{equation}
For $m\le-1$, we obtain
\begin{equation}
\label{eq:C_neg}
C_k(m,j) = \sum_{s=0}^{\lfloor j/k \rfloor} (-1)^{j-ks} \binom{|m|}{j-ks} \binom{|m|+s-1}{|m|-1} \,.
\end{equation}
A binomial coefficient is zero if its numerator is less than its denominator.  
\end{proposition}
\begin{proof}
We begin with $m\ge0$.
Observe that
\begin{equation}
\begin{split}
  (1+x+\dots+x^{k-1})^m &= \frac{(1-x^k)^m}{(1-x)^m}
  \\
  &= \biggl(\sum_{s=0}^m (-1)^s\binom{m}{s}x^{ks}\biggr) \biggl(\sum_{t=0}^\infty \binom{m+t-1}{m-1} x^t \biggr) \,.
\end{split}
\end{equation}
The coefficient of $x^j$ and is given by summing over tuples $(s,t)$ such that $t = j-ks$.
Since $t\ge0$, we must have $s \le \lfloor j/k \rfloor$. Hence
\begin{equation}
  C_k(m,j) = \sum_{s=0}^{\lfloor j/k \rfloor} (-1)^s\binom{m}{s} \binom{m+j-ks-1}{m-1} \,.
\end{equation}
This proves eq.~\eqref{eq:C_pos}.
Next we treat $m\le-1$.
The derivation is similar to that for eq.~\eqref{eq:C_pos}. For $m<0$, 
\begin{equation}
\begin{split}
  \frac{1}{(1+x+\dots+x^{k-1})^{|m|}} &= \frac{(1-x)^{|m|}}{(1-x^k)^{|m|}}
  \\
  &= \biggl(\sum_{t=0}^{|m|}(-1)^t\binom{|m|}{t}x^t\biggr) \biggl(\sum_{s=0}^\infty \binom{|m|+s-1}{|m|-1}x^{ks}\biggr) \,.
\end{split}
\end{equation}
The coefficient of $x^j$ and is given by summing over tuples $(s,t)$ such that $t = j-ks$.
Since $t\ge0$, we must have $s \le \lfloor j/k \rfloor$. Hence
\begin{equation}
C_k(m,j) = \sum_{s=0}^{\lfloor j/k \rfloor} (-1)^{j-ks} \binom{|m|}{j-ks} \binom{|m|+s-1}{|m|-1} \,.
\end{equation}
This proves eq.~\eqref{eq:C_neg}.
\end{proof}

\subsection{Polynomial coefficients}\label{sec:poly_coeff}
For $n>0$, the polynomial $\mathcal{F}_{n,k}(x)$ is given by
\begin{equation}
\label{eq:poly_pos1}
\mathcal{F}_{n,k}(x) = \sum_{h=0}^{\lfloor(k-1)(n-1)/k\rfloor} C_k(n-h-1, (k-1)(n-1)-hk)\,x^{(n-1)(k-1)-hk} \,.
\end{equation}
For $m\ge0$, observe that $C_k(m,j) = C_k(m, (k-1)m-j)$.
See Table \ref{tb:tbl_k4} for examples.
Hence for $n>0$, we may equivalently write
\begin{equation}
\label{eq:poly_pos}
\mathcal{F}_{n,k}(x) = \sum_{h=0}^{\lfloor(k-1)(n-1)/k\rfloor} C_k(n-h-1, h)\,x^{(n-1)(k-1)-hk} \,.
\end{equation}
The sum in eq.~\eqref{eq:poly_pos} is equivalent to that published by Hoggatt and Bicknell (\cite{HoggattBicknell_FibGen}, eq.~(4.2)).
For $n\le-(k-1)$, the polynomial $\mathcal{F}_{n,k}(x)$ is given by
\begin{equation}
\label{eq:poly_neg}
\mathcal{F}_{n,k}(x) = \sum_{h=0}^{\lfloor(|n|+1-k)/k\rfloor} C_k(-(1+h), |n|+1-k(1+h))\,x^{|n|+1-k(h+1)} \,.
\end{equation}
Some or all of the coefficients in the sum in eq.~\eqref{eq:poly_neg} may be zero.
We shall specify tight bounds for the sum in eq.~\eqref{eq:poly_neg} later,
including the criterion for a polynomial to vanish identically.
We obtain a more useful expression for $C_k(-(1+h),|n|+1-k(1+h))$ in eq.~\eqref{eq:poly_neg} as follows.
Recall $q$ and $r$ in eq.~\eqref{eq:q_r_def} and $|n|+1=kq+r$ for $n\le0$.
Substituting in eq.~\eqref{eq:C_neg} yields $\lfloor (|n|+1-k(1+h))/k \rfloor = \lfloor (k(q-h-1)+r)/k \rfloor = q-h-1$, whence
\begin{equation}
\label{eq:Cnk_neg_qr}    
\begin{split}
  C_k(-(1+h),|n|+1-k(1+h)) &= \sum_{s=0}^{q-h-1} (-1)^{k(q-s-h-1)+r} 
  \binom{1+h}{k(q-h-s-1)+r} \binom{h+s}{s} \,.
\end{split}
\end{equation}
We shall employ eq.~\eqref{eq:Cnk_neg_qr} below.
\begin{remark}\label{rem:high_power_neg}
For fixed $k\ge2$ and $n<0$, the highest power of $x$ in eq.~\eqref{eq:poly_neg} is attained when $h=0$ and is $|n|+1-k$.
Table \ref{tb:Fib_poly_neg} displays examples where the highest power $|n|+1-k$ is attained for some values of $k$ and $n$.
(There are also polynomials whose degree is lower than $|n|+1-k$.)
\end{remark}

\subsection{Additional structure of polynomials for $n>0$}\label{sec:poly_pos}
For $n>0$, it is well-known that
(i) the coefficients of all the terms in $\mathcal{F}_{n,k}(x)$ are positive;
(ii) the powers of $x$ have the form $x^{(k-1)(n-1)-hk}$, where $0\le h\le \lfloor(k-1)(n-1)/k\rfloor$;
(iii) the degree of $\mathcal{F}_{n,k}(x)$ is $(k-1)(n-1)$; and
(iv) the lowest exponent is given by the remainder $(k-1)(n-1)\!\!\!\mod k$
(this is $r_{n,k}$, see eq.~\eqref{eq:r_def}).
Let us treat $n\ge3$ below, because $\mathcal{F}_{1,k}(x)=1$ and $\mathcal{F}_{2,k}(x)=x^{k-1}$ are monomials.
The coefficient of the highest term in $\mathcal{F}_{n,k}(x)$ is known to be unity.
Using eq.~\eqref{eq:C_pos}, the coefficient of the second-highest term in eq.~\eqref{eq:poly_pos} is given by $h=1$ and is 
\begin{equation}
\label{eq:Cnk1_pos}
\begin{split}
C_k(n-2,1) &= \sum_{s=0}^{\lfloor 1/k \rfloor} (-1)^s\binom{n-2}{s} \binom{n-ks-2}{n-3} \quad (\Rightarrow\;s=0\;\textrm{only})
\\
&= \binom{n-2}{n-3}
\\
&= n-2 \,.
\end{split}
\end{equation}
The calculation is more complicated for $h\ge2$, but we can give a simple expression for the lowest term.
\begin{proposition}\label{prop:low_pos}
(The lowest term in $\mathcal{F}_{n,k}(x)$ for $n>0$.)
Recall $q$ and $r$ from eq.~\eqref{eq:q_r_def}, dropping the subscripts for brevity.
We again treat $n\ge3$ because $\mathcal{F}_{1,k}(x)=1$ and $\mathcal{F}_{2,k}(x)=x^{k-1}$ are monomials.
For $n\ge3$, the lowest term in $\mathcal{F}_{n,k}(x)$, say $L_{n,k}(x)$, is
\begin{equation}  
\label{eq:low_pos}
L_{n,k} = \binom{q+r}{q}\,x^r = \binom{q+r}{r}\,x^r \,.
\end{equation}
\end{proposition}
\begin{proof}
We employ an induction proof.
Recall the exponent of the lowest term is $r$.
Fix the value of $k$, then fix a value of $n$ and calculate $q$ and $r$ using eq.~\eqref{eq:q_r_def}.
Recall $n-1 = k(q+1)-r$ for $n>0$.
Equating powers of $x$ in the recurrence eq.~\eqref{eq:rec_Fnk_pos}
yields the following recurrence for the lowest term $L_{n,k}(x)$:
\begin{equation}  
\label{eq:rec_L_pos}
  L_{n,k}(x) = \sum_{i=0}^r x^i L_{n-k+r-i,k}(x) \,.
\end{equation}
The only terms which contribute to the sum have $q^\prime=q-1$,
so we suppose eq.~\eqref{eq:low_pos} is valid for all $n^\prime$ such that $q^\prime < q$ (and all $r^\prime \in[0,k-1]$).
Substitute into eq.~\eqref{eq:rec_L_pos} to obtain a sum $S$. We omit the common factor $x^r$. We obtain 
\begin{equation}  
\begin{split}
S &= \sum_{i=0}^{r} \binom{q-1+i}{q-1}
\;=\; \binom{q+r}{q} \,.
\end{split}  
\end{equation}
We easily verify eq.~\eqref{eq:low_pos} to be true for $n\in[1,k]$, i.e.~$q=0$ and $r\in[0,k-1]$.
The proof follows by induction on $n$, as we step up the values of $q$ and $r$.
\end{proof}

\setcounter{equation}{0}
\section{Polynomials for $n<0$}\label{sec:poly_neg}
\subsection{Identically vanishing polynomials}\label{sec:poly_van}
\begin{proposition}\label{prop:van_indices}
For $k\ge2$ and $n\le0$, the polynomial $\mathcal{F}_{n,k}(x)$ vanishes identically if and only if $q_{n,k}<r_{n,k}$.
This includes the set of initial values $|n|\in[0,k-2]$.
Dropping the subscripts on $q$ and $r$,
recall that $kq+r = 1-n$ for $n\le0$, hence $\mathcal{F}_{n,k}(x) = \mathcal{F}_{1 -(kq+r),k}(x)$.
Then $\mathcal{F}_{n,k}(x)$ vanishes identically for all tuples $(q,r)$ such that
$0 \le q < r \le k-1$.
Such tuples exist if and only if $0 \le q \le k-2$.
The number of such tuples is $k(k-1)/2$, including the initial conditions.
If $|n|\ge k^2-k-1$, the polynomials do not vanish identically.
\end{proposition}
The proof of Prop.~\ref{prop:van_indices} proceeds in two steps.
For $n<0$, let $\mathcal{G}_k(w)$ denote the generating function for the recurrence in eq.~\eqref{eq:rec_Fnk_neg}.
Expanding $\mathcal{G}_k(w)$ in a series in powers of $w$, we show that there are some values of $n$ 
for which a term in $w^{|n|}$ does not appear, i.e., the corresponding polynomial $\mathcal{F}_{n,k}(x)$ vanishes identically.
Next, we show that this is the exhaustive list of the vanishing polynomials.

The generating function is as follows (the overall minus sign is to obtain the correct initial value):
\begin{equation}  
\label{eq:genfcn_Fibpoly_neg}
\begin{split}  
  \mathcal{G}_k(w) &= -\frac{1/w}{1 -x^{k-1}/w -x^{k-2}/w^2 -\dots -x/w^{k-1} -1/w^k}
  \\
  &= \frac{w^{k-1}}{-w^k +x^{k-1}w^{k-1} +x^{k-2}w^{k-2} +\dots +xw +1}
  \\
  &= \frac{w^{k-1}}{(x^kw^k-1)/(xw-1)-w^k}
  \\
  &= \frac{w^{k-1}(1-xw)}{1 -(1+x^k -xw)w^k} \,.
\end{split}  
\end{equation}
Then $\mathcal{F}_{n,k}(x)$ equals the coefficient of $w^{|n|}$ in the Maclaurin series of $\mathcal{G}_k(w)$.
We show that certain powers of $w$ do not appear in that series.
The corresponding polynomials are thus identically zero.
We deduce the following from eq.~\eqref{eq:genfcn_Fibpoly_neg}:
\begin{equation}  
\label{eq:genfcn_Fibpoly_neg_w_pow}
\begin{split}  
  \mathcal{G}_k(w) &= w^{k-1}(1-xw)\sum_{s=0}^\infty w^{ks}(1+x^k -xw)^s
  \\
  &= w^{k-1}(1-xw)
  \\
  &\qquad +w^{2k-1}(1-xw)(1+x^k -xw)
  \\
  &\qquad +w^{3k-1}(1-xw)(1+x^k -xw)^2 
  \\
  &\qquad +w^{4k-1}(1-xw)(1+x^k -xw)^3 +\dots
  \\
  &= w^{k-1} -xw^k
  \\
  &\qquad +w^{2k-1}(1+x^k) -(2x+x^{k+1})w^{2k} +x^2w^{2k+1}
  \\
  &\qquad +w^{3k-1}(1+x^k)^2 +\dots -x^3w^{3k+2}
  \\
  &\qquad +w^{4k-1}(1+x^k)^3 +\dots +x^4w^{4k+3}
  \\
  &\qquad +\dots
\end{split}  
\end{equation}
We deduce several pertinent consequences from eq.~\eqref{eq:genfcn_Fibpoly_neg_w_pow}.
\begin{enumerate}
\item
The lowest power of $w$ which appears is $k-1$.
The initial values are $\mathcal{F}_{n,k}(x)=0$ for $n\in[-(k-2),0]$.
This is a block of $k-1$ consecutive indices.
\item
The two lowest nonzero terms are $w^{k-1}$ and $-xw^k$, hence $\mathcal{F}_{-(k-1),k}(x)=1$ and $\mathcal{F}_{-k,k}(x)=-x$.
This was observed in Table \ref{tb:Fib_poly_neg}.
\item
\emph{For $k\ge3$, there is a gap in the powers of $w$ from $k+1$ through $2k-2$.}
Hence all the polynomials $\mathcal{F}_{n,k}(x)$ vanish identically for $|n|\in[k+1,2k-2]$. This is a block of length $k-2$.
\item
The next nonvanishing term is in $w^{2k-1}$, hence $\mathcal{F}_{-(2k-1),k}=1+x^k$.
\item
The following term is in $w^{2k}$, hence $\mathcal{F}_{-2k,k}=-(2x+x^{k+1})=-x(x^k+2)$.
\item
The following terms are in $w^{2k+1}$ and $w^{3k-1}$. If $k=2$ the exponents are equal, else $3k-1>2k+1$.
\item
Hence for $k\ge3$, the term is given by $w^{2k+1}$ alone, hence $\mathcal{F}_{(-2k-1),k}=x^2$.
\item
Also for $k\ge3$, the next nonvanishing term is in $w^{3k-1}$, hence $\mathcal{F}_{-(3k-1),k}=(1+x^k)^2$.
\item
  By the same reasoning, there is a gap in the powers of $w$ from $2(k+1)$ through $3k-2$.
  All the polynomials $\mathcal{F}_{n,k}(x)$ vanish identically for $|n|\in[2k+2,3k-2]$. This is a block of length $k-3$.
\item
  Proceeding in this way, one sees that a block of vanishing polynomials lies in the interval
  $|n|\in[s(k+1),(s+1)k-2]$. This is a block of length $k-s-1$. 
  The last block has length $1$, hence $s\in[0,k-2]$.
\item
  The total number of such identically vanishing polynomials is $(k-1) +(k-2) +\dots +1 = k(k-1)/2$.
\end{enumerate}
For $n\in[-(k-2),0]$, the absent indices are given by $q_{n,k}=0$ and $r_{n,k}\in[1,k-1]$.
For $n\in[-(2k-2),-(k-1)]$, the absent indices are given by $q_{n,k}=1$ and $r_{n,k}\in[2,k-1]$, etc.
That is to say, the set of absent indices is given by $q_{n,k}<r_{n,k}$.
However, there may be other values of $|n|$ for which the terms sum to zero.
We next show that this is not so.

\begin{proposition}\label{prop:low_neg}
(The lowest term in $\mathcal{F}_{n,k}(x)$ for $n<0$.)
Again recall $q$ and $r$ in eq.~\eqref{eq:q_r_def}.
For $n<0$, the expression for the lowest term $L_{n,k}(x)$ in the polynomial $\mathcal{F}_{n,k}(x)$ is
\begin{equation}  
\label{eq:low_neg}
  L_{n,k}(x) = (-1)^r \binom{q}{r}\,x^r \,.
\end{equation}
The binomial coefficient vanishes if $q<r$.
\end{proposition}
\begin{proof}
We employ an induction proof.
First fix the value of $k$.
Then fix a value of $n$ and calculate $q$ and $r$ using eq.~\eqref{eq:q_r_def}.
Recall $|n|+1 = kq+r$.
Equating powers of $x$ in the recurrence eq.~\eqref{eq:rec_Fnk_neg}
yields the following recurrence for the lowest term $L_{n,k}(x)$:
\begin{equation}  
\label{eq:rec_L_neg}
  L_{n,k}(x) = L_{n+k,k}(x) -\sum_{i=1}^r x^i L_{n+i,k}(x) \,.
\end{equation}
We suppose eq.~\eqref{eq:low_neg} is valid for all $0 < |n^\prime| < |n|$,
i.e.~tuples $(q^\prime,r^\prime)$, where the following two conditions are satisfied.
First, $q^\prime < q$ (and all $r^\prime \in[0,k-1]$).
Next, $q^\prime = q$ and $r^\prime \in[0,r-1]$.
The term $L_{n+k,k}(x)$ has $q^\prime=q-1$ and $r^\prime=r$.
The other terms have $q^\prime=q$ and $r^\prime < r$.
Substitute into eq.~\eqref{eq:rec_L_neg} to obtain a sum $S^\prime$. We omit the common factor $x^r$, whence
\begin{equation}  
\begin{split}  
  S^\prime &= (-1)^r\binom{q-1}{r} -\sum_{i=1}^r (-1)^{r-i}\binom{q}{r-i}
  \\
  &= (-1)^r\binom{q-1}{r} +(-1)^r\binom{q-1}{r-1}
  \\
  &= (-1)^r\binom{q}{r} \,.
\end{split}  
\end{equation}
The result is true by definition for $|n|\in[0,k-2]$ because that is the set of the initial values.
We easily verify eq.~\eqref{eq:low_neg} to be true for the nontrivial interval $|n|\in[k-1,2k-2]$, i.e.~$q=1$ and $r\in[0,k-1]$.
The proof follows by induction on $|n|$, as we step up the values of $q$ and $r$.
\end{proof}

\begin{proof}\label{proof:proof_van_indices}
(Completion of proof of Prop.~\ref{prop:van_indices}.)
Using the generating function in eq.~\eqref{eq:genfcn_Fibpoly_neg}, we showed that 
$\mathcal{F}_{n,k}(x)$ vanishes identically if $q_{n,k} < r_{n,k}$.
Prop.~\ref{prop:low_neg} showed that
$\mathcal{F}_{n,k}(x)$ is nonvanishing if $q_{n,k} \ge r_{n,k}$.
This establishes the ``if and only if'' claim.
The polynomials do not vanish identically for any $n<0$ such that $q_{n,k}\ge k-1$.
It is easily verified that this means $|n| \ge k^2-k-1$.
\end{proof}
\begin{remark}
We actually proved additional properties of the polynomials not stated in Prop.~\ref{prop:van_indices}.
\begin{enumerate}
\item
  From the enumerated list following eq.~\eqref{eq:genfcn_Fibpoly_neg_w_pow},
  the first nonvanishing polynomial (after a previous vanishing block ends) is $\mathcal{F}_{-(sk-1),k}(x)=(1+x^k)^{s-1}$, for $s\in[1,k-2]$.
\item
  Also from the same enumerated list,
  the last nonvanishing polynomial (before the next vanishing block begins) is $\mathcal{F}_{1-s(k+1),k}(x)=(-1)^s x^s$, for $s\in[0,k-2]$.
  For $s=0$, this is the initial value $\mathcal{F}_{1,k}(x)=1$.
\item
  Note that eq.~\eqref{eq:low_neg} presented the complete expression for the lowest term in a nonvanishing polynomial $\mathcal{F}_{n,k}(x)$ for $n<0$.  
\item
  Also, eq.~\eqref{eq:low_pos} presented the complete expression for the lowest term in the polynomial $\mathcal{F}_{n,k}(x)$ for $n>0$
  (although this is not relevant to the proof of Prop.~\ref{prop:van_indices}).
\end{enumerate}
\end{remark}

\subsection{Vanishing polynomial coefficients}\label{sec:van_coeff}
As opposed to polynomials which vanish identically,
for $k\ge3$ and $n\le0$ there are also polynomials which are not identically zero but some, though not all, of the coefficients are zero.
This also does not happen for $n>0$.
An example is the Tribonacci polynomial $T_{-14}(x) = 1 +3x^6 +4x^9 +x^{12}$.
The term in $x^3$ is absent.
For $(n,k)=(-66,5)$, the Pentanacci polynomial has \emph{two} terms whose coefficients are zero:
$P_{-66}(x) = x^2(78 +3300x^{15} +11880x^{20} +\dots)$.
The terms in $x^5$ and $x^{10}$ are absent, in the sum in the parentheses.
This is the only instance found to date where there is more than one vanishing coefficient, in a polynomial which is not identically zero.  
Such ``vanishing coefficients'' have been found for all tested odd values of $k$ (there are no examples for even $k$).
There appears to be no simple pattern to determine the vanishing coefficients.

\subsection{Multinomial sum for negative indices}\label{sec:sum_neg}
Recall the generating function in eq.~\eqref{eq:genfcn_Fibpoly_neg}.
Let us derive a multinomial sum for the polynomial $\mathcal{F}_{n,k}(x)$ for $n<0$.
We process $\mathcal{G}_k(w)$ in eq.~\eqref{eq:genfcn_Fibpoly_neg} as follows:
\begin{equation}  
\begin{split}  
  \mathcal{G}_k(w) &= \frac{w^{k-1}}{1 +xw +\dots +x^{k-1}w^{k-1} -w^k}
  \\
  &= w^{k-1}\sum_{s=0}^\infty (-1)^s(xw +\dots +x^{k-1}w^{k-1} -w^k)^s
  \\
  &= \sum_{s=0}^\infty \sum_{j_1+j_2+\dots+j_k=s} (-1)^{s-j_k}\frac{s!}{j_1!j_2!\dots j_k!}
  \,w^{k-1+\sum_{i=1}^k ij_i} \,x^{\sum_{i=1}^{k-1} ij_i} \,.
\end{split}  
\end{equation}
Then $\mathcal{F}_{n,k}(x)$ equals the coefficient of $w^{|n|}$.
Set $|n| = k-1+\sum_{i=1}^k ij_i$, then $\sum_{i=1}^{k-1} ij_i = |n|+1-k(1+j_k)$.
Also define $J = \sum_{i=1}^k j_i$. We obtain
\begin{equation}  
\label{eq:multsum_neg}    
  \mathcal{F}_{n,k}(x) = \sum_{j_1+2j_2+\dots +kj_k=|n|+1-k} (-1)^{J-j_k}\frac{J!}{j_1!j_2!\dots j_k!} \,x^{|n|+1-k(1+j_k)} \,.
\end{equation}
Curiously, for fixed $k$ and $n$, the exponent of $x$ depends only the \emph{last} variable $j_k$.
Kuhapatanakul and Sukruan derived an equivalent sum for the Tribonacci numbers with negative indices (\cite{Kuhapatanakul_Sukruan}, eq.~(6)),
but see also the unnumbered equation before their eq.~(6), setting $a=x^2$, $b=x$, and $c=1$ yields the Tribonacci polynomials.
They did not remark that the sum is identically zero for some values of $n$.

\subsection{Additional structure of polynomials for $n<0$}\label{sec:add_structpoly}
\begin{proposition}\label{prop:deg}
For $k\ge2$ and $n\le0$, the degree $d_{n,k}$ of the polynomial $\mathcal{F}_{n,k}(x)$ is 
\begin{equation}
\label{eq:deg_neg}    
d_{n,k} = \begin{cases} |n|+1-k &\qquad (r_{n,k}=0)\,, \\ |n|+1-kr_{n,k} &\qquad (r_{n,k}\in[1,k-1])\,.  \end{cases} 
\end{equation}
The polynomial vanishes identically if $d_{n,k}<0$.
The highest term in a nonvanishing polynomial, say $H_{n,k}(x)$, is 
\begin{equation}
\label{eq:high_term_neg}    
H_{n,k}(x) = \begin{cases} x^{|n|+1-k} &\qquad (r_{n,k}=0)\,,\phantom{\biggl|}
  \\ \displaystyle
  (-1)^{r_{n,k}} \binom{q_{n,k}-1}{r_{n,k}-1}x^{|n|+1-kr_{n,k}} &\qquad (r_{n,k}\in[1,k-1])\,. \end{cases} 
\end{equation}
\end{proposition}
\begin{proof}
It is easily seen that eq.~\eqref{eq:deg_neg} yields $d_{n,k}<0$ for $n\in[-(k-2),0]$ (the initial conditions).
Hence we treat $|n|\ge k-1$ below.
We prove eq.~\eqref{eq:deg_neg} by deriving the highest terms in eq.~\eqref{eq:high_term_neg} first, then deduce the degree.  
For fixed $k$ and $n$, the highest power of $x$ is given by minimizing the value of $h$.
Dropping the subscripts on $q$ and $r$, recall $kq+r=|n|+1$ and $q\ge1$ for $|n|\ge k-1$.
First suppose $r=0$, then eq.~\eqref{eq:Cnk_neg_qr} yields the following:
\begin{equation}
\begin{split}
  C_k(-(1+h),|n|+1-k(1+h)) &= \sum_{s=0}^{q-h-1} (-1)^{k(q-s-h-1)} 
  \binom{1+h}{k(q-h-s-1)} \binom{h+s}{s} \,.
\end{split}
\end{equation}
The highest power of $x$ is obtained by minimizing $h$.
Set $h=0$, then we obtain 
\begin{equation}
\label{eq:Ck_high0}  
\begin{split}
C_k(-1,|n|+1-k) &= \sum_{s=0}^{q-1} (-1)^{k(q-s-1)} \binom{1}{k(q-s-1)} \binom{s}{s} \,.
\end{split}
\end{equation}
A solution exists only if $q-s-1=0$, i.e.~$s=q-1$, else the first binomial coefficient vanishes.
Then only the last term in the sum in eq.~\eqref{eq:Ck_high0} contributes, whence
\begin{equation}
C_k(-1,|n|+1-k) = \binom{1}{0} \binom{q-1}{q-1} = 1 \,.
\end{equation}
The corresponding power of $x$ is $x^{|n|+1-k}$.
This proves the case $r=0$ in eq.~\eqref{eq:high_term_neg}.
Next suppose $1 \le r \le k-1$.
Observe in eq.~\eqref{eq:Cnk_neg_qr} that if $k(q-h-s-1)<0$, then
$k(q-h-s-1)+r \le -k+r < 0$, so the first binomial coefficient will vanish.
Hence we must have $k(q-h-s-1)\ge0$.
Then a solution cannot exist if $1+h<r$, because the first binominal coefficient will vanish (numerator less than denominator).
Hence the smallest allowed value of $h$ is $h=r-1$.
Then we obtain 
\begin{equation}
\label{eq:Ck_high1}
C_k(-r,|n|+1-kr) = \sum_{s=0}^{q-r} (-1)^{k(q-s-r)+r} \binom{r}{k(q-r-s)+r} \binom{s+r-1}{s} \,.
\end{equation}
If $q-r<0$, the polynomial vanishes identically.
If $q-r\ge0$, we must have $q-r-s=0$, i.e.~$s=q-r$, else the first binomial coefficient in eq.~\eqref{eq:Ck_high1} will vanish.
Then only the last term in the sum in eq.~\eqref{eq:Ck_high1} contributes, whence
\begin{equation}
C_k(-r,|n|+1-kr) = (-1)^r \binom{r}{r} \binom{q-1}{q-r} = (-1)^r \binom{q-1}{r-1} \,.
\end{equation}
The corresponding power of $x$ is $x^{|n|+1-kr}$.
This proves the case $r>0$ in eq.~\eqref{eq:high_term_neg}.
The corresponding powers of $x$, for all $r\in[0,k-1]$, yield the degree $d_{n,k}$ in eq.~\eqref{eq:deg_neg}.
\end{proof}
\begin{remark}
From eq.~\ref{eq:deg_neg}, the degree does not increase monotonically with $|n|$ for $n<0$.
Recall from eq.~\eqref{eq:multsum_neg} that the highest attainable power of $x$ is $|n|+1-k$.
From eq.~\eqref{eq:high_term_neg}, that value is attained when $r_{n,k}=0$ or $1$.
\end{remark}
\begin{remark}
For fixed $k$, if the value of $n$ decreases by $k$ (i.e.~$|n|$ increases by $k$), the degree of the polynomial increases by $k$.
The degree can be visualized as a set of parallel lines with slope $1$, indexed by $r_{n,k}$.
Figure \ref{fig:deg} displays a plot of the degree $d_{n,k}$ of a polynomial for $k=5$ and $n>0$ (dotdash line) and $n<0$ (solid line and circles).
Recall for $n>0$ the degree is $d_{n,k}=(k-1)(n-1)$, hence the plot is a straight line with slope $k-1$.
For $k\gg2$ and $|n|\gg1$, the degree is (much) higher for positive $n$.
\end{remark}

\begin{proposition}\label{prop:sec_highest_neg}
For $k\ge2$ and $n\le0$, the second highest term, say $S_{n,k}(x)$, in a nonvanishing polynomial which is also not a monomial, is as follows:
\begin{equation}
\label{eq:sec_term_neg}    
S_{n,k}(x) = \begin{cases}
(2q_{n,k}-3)x^{|n|+1-2k} &\quad (r_{n,k}=0,\, k=2)\,,\phantom{\biggl|} \\
(q_{n,k}-1)x^{|n|+1-2k} &\quad (r_{n,k}=0,\, k\ge3)\,,\phantom{\biggl|} 
  \\ \displaystyle (-1)^{r_{n,k}} (r_{n,k}+1)\binom{q_{n,k}-1}{r_{n,k}}x^{|n|+1-k(r_{n,k}+1)} &\quad (r_{n,k}\in[1,k-1])\,. \end{cases} 
\end{equation}
To avoid a monomial, we require $|n|\ge 2k-1$ for $r_{n,k}=0$ and $|n|\ge k(r+1)-1$ for $r_{n,k}\in[1,k-1]$.
\end{proposition}
\begin{proof}
This is a followup of eq.~\eqref{eq:high_term_neg} and the derivation follows the same reasoning.
We again drop the subscripts on $q$ and $r$.
First suppose $r=0$.
For the highest power, we set $h=0$, hence for the second highest power we set $h=1$.
Then eq.~\eqref{eq:Cnk_neg_qr} yields the following:
\begin{equation}
\label{eq:Ck_low0}
\begin{split}
  C_k(-2,|n|+1-2k) &= \sum_{s=0}^{q-2} (-1)^{k(q-s-2)} \binom{2}{k(q-s-2)} \binom{s+1}{s} \,.
\end{split}
\end{equation}
A solution exists only if $k(q-s-2)\le2$, else the first binomial coefficient in eq.~\eqref{eq:Ck_low0} vanishes.
If $k\ge3$, we must have $s=q-2$, whence $k(q-s-2)=0$.
Then only the last term in the sum in eq.~\eqref{eq:Ck_low0} contributes and we obtain 
\begin{equation}
C_k(-2,|n|+1-2k) = \binom{2}{0} \binom{q-1}{q-2} = q-1 \,.
\end{equation}
If $k=2$, there is also a solution where $s\ge q-3$, because $k(q-s-2) = 2$ if $s=q-3$ and $k(q-s-2) = 0$ if $s=q-2$. 
Then eq.~\eqref{eq:Ck_low0} yields the following:
\begin{equation}
\begin{split}
  C_k(-2,|n|+1-2k) &= \sum_{s=q-3}^{q-2} (-1)^{2(q-s-2)} \binom{2}{2(q-s-2)} \binom{s+1}{s}
  \\
  &= \binom{2}{2} \binom{q-2}{q-3} + \binom{2}{0} \binom{q-1}{q-2}
  \\
  &= 2q-3 \,.
\end{split}
\end{equation}
For all $k\ge2$, the corresponding power of $x$ is $x^{|n|+1-2k}$.
This proves the first two cases in eq.~\eqref{eq:sec_term_neg}.
Next suppose $1 \le r \le k-1$. 
Employing the same reasoning used for eq.~\eqref{eq:high_term_neg} for $r>0$,
now we deduce the lowest value of $h$ for which a solution exists is $h=r$. We obtain the following:
\begin{equation}
\label{eq:Ck_low1}
\begin{split}
  C_k(-(1+r),|n|+1-k(1+r)) &= \sum_{s=0}^{q-r-1} (-1)^{k(q-s-r-1)+r} 
  \binom{r+1}{k(q-r-s-1)+r} \binom{r+s}{s} \,.
\end{split}
\end{equation}
If $q<r+1$ there is no second highest term, hence we assume $q\ge r+1$.
We must have $q-r-s-1=0$, i.e.~$s=q-r-1$, else the first binomial coefficient in eq.~\eqref{eq:Ck_low1} vanishes.
Then eq.~\eqref{eq:Ck_low1} has a solution for all $k\ge2$.
Only the last term in the sum in eq.~\eqref{eq:Ck_low1} contributes and we obtain the following:
\begin{equation}
\begin{split}
  C_k(-(1+r),|n|+1-k(1+r)) &= (-1)^r \binom{r+1}{r} \binom{q-1}{q-r-1}
  \\
  &= (-1)^r (r+1) \binom{q-1}{r} \,.
\end{split}
\end{equation}
The corresponding power of $x$ is $x^{|n|+1-k(r+1)}$.
\end{proof}

\setcounter{equation}{0}
\section{Skolem problem}\label{sec:Skolem}
\subsection{Statement of problem}
The \textit{Skolem problem} is as follows:
let $u$ denote the set $\{ u_n,\, n\in\mathbb{Z}\}$ of a sequence of numbers which satisfy a linear recurrence with constant coefficients.
Then find the set $Z(u) = \{n\in\mathbb{Z}:\, u_n=0\}$.
The cardinality of $Z(u)$, if it is finite, is called the \textit{zero-multiplicity} of the sequence $u$.
There is no known general algorithm to find $Z(u)$, or even its cardinality.
However, Skolem \cite{Skolem1933} proved the following result:
if the coefficients of a linear recurrence sequence are rational,
then the set $Z(u)$ is a union of finitely many arithmetic progressions together with a (possibly empty) finite set.

\subsection{Skolem problem for $k$-Fibonacci numbers}
For the $k$-Fibonacci numbers $F_{n,k}$, the Skolem problem was solved by
Garc{\'i}a, G{\'o}mez and Luca \cite{GarciaGomezLuca2020,GarciaGomezLuca2024}.
For fixed $k\ge2$, they defined $\zeta_k$ as the zero-multiplicity of the $k$-Fibonacci sequence.
In (\cite{GarciaGomezLuca2020}, Corollary 1.2), they showed that $\zeta_2=1$, $\zeta_3=4$ and $\zeta_k = k(k-1)/2$ for $k\in[4,500]$.
They also showed that $\zeta_k \ge k(k-1)/2$ for all $k\ge2$.
Later in (\cite{GarciaGomezLuca2024}, Theorem 2), they showed that $\zeta_k = k(k-1)/2$ for $k\ge 501$,
i.e.~$\zeta_k = k(k-1)/2$ for all $k\ge4$.
For fixed $k\ge2$, they also determined the values of $n$ for which $F_{n,k}$ vanishes.

Our findings are related the Skolem problem for the $k$-Fibonacci numbers $F_{n,k}$ as follows.
Observe that for $x=1$, the recurrences in eqs.~\eqref{eq:rec_Fnk_num} and \eqref{eq:rec_Fnk_pos} coincide.
\begin{enumerate}
\item
If a $k$-Fibonacci polynomial vanishes identically, then the corresponding $k$-Fibonacci number equals zero.
For fixed $k\ge2$, the set of indices for which the $k$-Fibonacci polynomials vanish identically forms a union of arithmetic progressions.
The indices we obtained agree with those found by Garc{\'i}a, G{\'o}mez and Luca.
\item
If a nonvanishing $k$-Fibonacci polynomial has a root $x=1$, then the corresponding $k$-Fibonacci number also equals zero.
The Tribonacci polynomial $T_{-17}(x) = 1 -5x^3 -6x^6 +4x^9 +5x^{12} +x^{15}$ has a root $x=1$ but is not identically zero.
The special case $T_{-17}(1)=0$ is an example of the finite set mentioned by Skolem
and also agrees with the findings by Garc{\'i}a, G{\'o}mez and Luca.
\item
The results of Garc{\'i}a, G{\'o}mez and Luca prove that $T_{-17}(x)$ is the only nonvanishing $k$-Fibonacci polynomial with a root $x=1$.
\end{enumerate}

\setcounter{equation}{0}
\section{Factorization of polynomials}\label{sec:poly_fac}
For fixed $k\ge2$ and all $n\in\mathbb{Z}$, we have proved that the
lowest power of $x$ which appears in a nonvanishing polynomial is $x^{r_{n,k}}$
(see eqs.~\eqref{eq:low_pos} and \eqref{eq:low_neg}).
We have also seen that the powers of $x$ increase in steps of $h$.
Hence for fixed $k\ge2$ and all $n\in\mathbb{Z}$, 
the polynomial $\mathcal{F}_{n,k}(x)$ has the structure
\begin{equation}  
\label{eq:structpoly}
\mathcal{F}_{n,k}(x) = x^{r_{n,k}} P_{n,k}(x^k) \,.
\end{equation}
For a nonvanishing polynomial $\mathcal{F}_{n,k}(x)$,
(i) the polynomial $P_{n,k}(x^k)$ is a function of $x^k$ only and $P_{n,k}(0) \ne 0$,
and (ii) if $\mathcal{F}_{n,k}(0)=0$, the root $x=0$ has multiplicity $r_{n,k}$.
For $n>0$, the above results were proved by Hoggatt and Bicknell \cite{HoggattBicknell_FibGen}.
(See also \cite{HeSimonRicci_num}.)
\begin{remark}
The nonzero roots have a $k$-fold rotational symmetry around the origin in the complex plane.
Kaymak and {\"O}zg{\"u}r \cite{Kaymak_Ozgur} stated matter succinctly (although they treated only $n>0$).
Let $\omega_k$ be a primitive $k^{th}$ root of unity:
the roots of $\mathcal{F}_{n,k}(\omega_kx)$ and $\mathcal{F}_{n,k}(x)$ are identical.
\end{remark}

We can factorize the polynomials in more detail than eq.~\eqref{eq:structpoly}.
It is well-known that for $n \in [1,k+1]$, the $k$-Fibonacci polynomials are as follows:
\begin{equation}
\label{eq:Fnk_simple}
\mathcal{F}_{n,k}(x) = \begin{cases} 1 & \qquad (n=1)\,, \\ x^{k+1-n}(x^k+1)^{n-2} &\qquad (n \in[2,k+1])\,. \end{cases}
\end{equation}
It does not seem to be recognized that a similar pattern
to eq.~\eqref{eq:Fnk_simple}
extends to all higher values of $n$.
We make a more precise statement as follows.
We have already noted the prefactor $x^{r_{n,k}}$ in eq.~\eqref{eq:structpoly}.
Here we remark on the structure of the polynomial $P_{n,k}(x^k)$ in eq.~\eqref{eq:structpoly}.
Observe that for $n\in[1,k+1$], $P_{n,k}(x^k)=(x^k+1)^{\rho_{n,k}}$, where $\rho_{n,k}=0$ for $n=1$ and $\rho_{n,k}=n-2$ for $n\in[2,k+1]$.
We extend the definition of $\rho_{n,k}$ as follows for all $n\ge1$:
\begin{equation}
\label{eq:rho_def_pos}
\rho_{n,k} = ((n-2)\!\!\!\mod(k+1))\!\!\!\mod k \,.
\end{equation}
The value of $\rho_{n,k}$ repeats with a period $k+1$.
Curiously, $\rho_{n,k}$ vanishes for \emph{two} values of $n$ in an interval of length $k+1$,
viz.~$n=s(k+1)+1$ and $n=s(k+1)+2$, for $s\ge0$.
We extend the definition of $\rho_{n,k}$ to $n\le0$ as follows:
\begin{equation}
\label{eq:rho_def_neg}
\begin{split}
\rho_{n,k} &= (((k+1)|n|-|n|-2)\!\!\!\mod(k+1))\!\!\!\mod k
\\
&= ((k|n|-2)\!\!\!\mod(k+1))\!\!\!\mod k \qquad\qquad\qquad (n\le0) \,.
\end{split}
\end{equation}
\begin{proposition}\label{prop:poly_extra_struct}
For $k\ge 2$ and $n\in\mathbb{Z}$, the $k$-Fibonacci polynomial $\mathcal{F}_{n,k}(x)$ has the form below:
\begin{equation}
\label{eq:Fnk_factor}
\mathcal{F}_{n,k}(x) = x^{r_{n,k}}(x^k+1)^{\rho_{n,k}}Q_{n,k}(x^k) \,.
\end{equation}
\end{proposition}
\begin{proof}
We have already established the prefactor of $x^{r_{n,k}}$.
Set $\mathcal{F}_{n,k}(x) = A_{n,k}(x)/x^n$, then the recurrence for $A_{n,k}(x)$ is as follows:
\begin{equation}  
A_{n,k}(x) = x^k\bigl(A_{n-1,k}(x) +A_{n-2,k}(x) +\dots +A_{n-k,k}(x)\bigr) \,.
\end{equation}
Set $x^k=-1$, then we obtain the following:
\begin{equation}  
\label{eq:rec_A}
A_{n,k}(x) = -\bigl(A_{n-1,k}(x) +A_{n-2,k}(x) +\dots +A_{n-k,k}(x)\bigr) \,.
\end{equation}
The characteristic polynomial $C(z)$ of the recurrence in eq.~\eqref{eq:rec_A} is as follows:
\begin{equation}  
C(z) = z^k +z^{k-1} +\dots +1 = \frac{z^{k+1}-1}{z-1} \,.
\end{equation}
The roots of $C(z)$ are the $(k+1)^{th}$ roots of unity excluding $z=1$,
say $z_j = \omega_{k+1}^j$, $j\in[1,k]$, and $\omega_{k+1}$ is a primitive $(k+1)^{th}$ root of unity.
Writing a Binet-style formula $A_{n,k}(x) = \sum_j c_jz_j^n$, where the $c_j$ are a set of coefficients chosen to fit the initial values,
it follows immediately that $A_{n+k+1,k}(x) = A_{n,k}(x)$, i.e., $A_{n,k}(x)$ repeats with a period $k+1$.
It follows that $\mathcal{F}_{n+k+1,k}(x) = (1/x^{k+1})\mathcal{F}_{n,k}(x) = (-1/x)\mathcal{F}_{n,k}(x)$.
It also follows that if $\mathcal{F}_{n,k}$ contains a factor $(x^k+1)$ with multiplicity $m_{n,k}$, then
$\mathcal{F}_{n+k+1,k}$ also contains a factor $(x^k+1)$ with the same multiplicity $m_{n,k}$.
We know that $\mathcal{F}_{n,k}$ indeed contains a factor $(x^k+1)$ with the multiplicity $\rho_{n,k}$ for $n\in[1,k+1]$,
whence the result follows.
\end{proof}
For a nonvanishing polynomial,
$Q_{n,k}(x^k)$ is a polynomial in $x^k$ only and $Q_{n,k}(x^k)$ is not divisible by $x^k+1$.
Also, $Q_{n,k}(0)\ne0$ and
from Propositions \ref{prop:low_pos} and \ref{prop:low_neg}, we deduce that
\begin{equation}  
  Q_{n,k}(0) = \begin{cases} \displaystyle \binom{q_{n,k}+r_{n,k}}{r_{n,k}} &\qquad (n>0)\,,\phantom{\Biggl|}
    \\ \displaystyle (-1)^{r_{n,k}}\binom{q_{n,k}}{r_{n,k}} &\qquad (n\le0)\,. \end{cases}
\end{equation}
If $k$ is odd, then $x=-1$ is a real (integer) root of $\mathcal{F}_{n,k}(x)$, with multiplicity $\rho_{n,k}$. 
Observe that the value of $r_{n,k}$ repeats with a period $k$ but the value of $\rho_{n,k}$ repeats with a period $k+1$,
so they do not remain in sync as $n$ increases.
For $n>0$, Table \ref{tb:Tri_Quad_poly}
displays examples of the factorization in eq.~\eqref{eq:Fnk_factor}
for the Tribonacci polynomials $T_n(x)$ and the Quadranacci polynomials $Q_n(x)$
for $n\in[1,2(k+1)]$, respectively.
Examples are also tabulated for the Tribonacci polynomials $T_n(x)$ and the Quadranacci polynomials $Q_n(x)$ in Table \ref{tb:Tri_Quad_poly_neg}
for $-10 \ge n\ge -19$,
including a vanishing polynomial at $n=-10$ for $Q_n(x)$.
The corresponding values of $\rho_{n,3}$ and $\rho_{n,4}$ are also tabulated in Tables \ref{tb:Tri_Quad_poly} and \ref{tb:Tri_Quad_poly_neg}.

\setcounter{equation}{0}
\section{Properties of polynomial roots}\label{sec:poly_roots}
For $n>0$, the structure of the polynomial in eq.~\eqref{eq:Fnk_factor} was correctly conjectured by He, Simon and Ricci (\cite{HeSimonRicci_num}, Conjecture 1),
in which they also correctly stated the pattern of the roots listed above.
Unfortunately their values for the exponents $r_{n,k}$ and $\rho_{n,k}$ are not always correct.
They also stated the roots extend outwards from the origin (in ``stars'').
This is true for $n>0$ but not necessarily so for $n<0$.
Argand diagram plots of the roots of $\mathcal{F}_{-40,3}(x)$ and $\mathcal{F}_{40,8}(x)$ are displayed in Figure \ref{fig:roots_poly}.
Compare the plots of the roots for $n=\pm40$ in Figure \ref{fig:roots_poly}.
For $n=-40$, one of the ``stars'' has positive and negative real roots, which straddle the origin (and two complex roots).
The other two ``stars'' also straddle the origin.
\begin{remark}\label{rem:root_k2}
  For $k=2$ and $n>0$, Hoggatt and Bicknell \cite{HoggattBicknell_roots} 
  showed that the roots of the Fibonacci polynomials $F_n(x)$ are pure imaginary
  (including the origin, if it is a root) and given by $2i \cos(j\pi/n)$, $j\in[1,n-1]$,
  and all the roots are simple.
\end{remark}
\begin{proposition}\label{prop:realroots}
(The nonzero real roots of $\mathcal{F}_{n,k}(x)$ for $n\ge3$.)
Recall $\mathcal{F}_{1,k}(x)=1$ and $\mathcal{F}_{2,k}(x)=x^{k-1}$, hence we exclude them and treat only $n\ge3$. Then
for even $k$, $\mathcal{F}_{n,k}(x)$ has no nonzero real roots and
for odd $k$, $\mathcal{F}_{n,k}(x)$ has no positive real roots but may have negative real roots.
\end{proposition}
\begin{proof}
We employ Descartes' rule of signs.
All the coefficients in $\mathcal{F}_{n,k}(x)$ are positive hence there are no positive real roots.
For negative real roots, replace $x$ by $-x$.
If $k$ is even, then $P_{n,k}((-x)^k)= P_{n,k}(x^k)$, hence $\mathcal{F}_{n,k}(x)$ has no negative real roots,
i.e.~if $n\ge3$ and $k$ is even, $\mathcal{F}_{n,k}(x)$ has no nonzero real roots.
If $k$ is odd, then $P_{n,k}((-x)^k)= P_{n,k}(-x^k)$,
hence successive coefficients have alternating signs.
If the number of sign changes is odd, there is an odd number of negative real roots, i.e.~at least one.
If the number of sign changes is even, there are zero else an even number of negative real roots.
Examples are the Tribonacci polynomials
$T_3(x) = x(1 +x^3)$ (root $x=-1$, multiplicity $1$), 
$T_4(x) = 1 +2x^3 +x^6$ (root $x=-1$, multiplicity $2$), and
$T_5(x) = x^2(3 +3x^3 +x^6)$, which has no nonzero real roots.
\end{proof}
\begin{remark}\label{rem:real_roots}
  Numerical calculations reveal the following findings.
\begin{enumerate}
\item
  For $n<0$ and even $k$, numerical calculations and symbolic manipulations indicate that the coefficients in a nonvanishing polynomial have no sign changes.
  This has been verified for all even $k$ such that $2 \le k \le 100$ and $|n|\le1000$.
  The fact that there are vanishing polynomials proves that the terms in the recurrence can cancel to zero.
  However, if they do not, then all the coefficients in the polynomial have the same sign.
  The reason for this is not known.  
  This has the consequence that, for even $k$ and all $n\in\mathbb{Z}$, the polynomial $\mathcal{F}_{n,k}(x)$ has no nonzero real roots.
\item
  For $n<0$ and odd $k$, the polynomials may have positive and/or negative real roots.
  See the example of $\mathcal{F}_{-40,3}(x)$ in Figure \ref{fig:roots_poly}.
A numerical search for nonzero integer roots reveals that $x=\pm1$ are the only cases.
The root $x=-1$ was treated above.
For $x=1$, the only solution is the Tribonacci polynomial $T_{-17}(x)$, as noted in Sec.~\ref{sec:Skolem}.
\item
For $n>0$ and odd $k$, numerical studies indicate that
all the real roots of $\mathcal{F}_{n,k}(x)$ lie in the interval $[-1,0]$.
The endpoints $-1$ and $0$ are respectively the lowest and highest values of a real root, and are attained for some values of $n$, as noted above.
For example, consider the Tribonacci polynomial 
\begin{equation}
  T_{10}(x) = 1 +16x^3 +45x^6 +50x^9 +28x^{12} +8x^{15} +x^{18} \,.
\end{equation}
It has two real roots, approximately $-0.862794$ and $-0.427835$, which both lie in the interval $(-1,0)$.

\item
For $n<0$ and odd $k$, a numerical search for real roots other than $0$ and $\pm1$ reveals that 
the roots are either positive or else less than $-1$.
There are no real roots in the interval $(-1,0)$.
For example, consider the Tribonacci polynomial 
\begin{equation}
  T_{-20}(x) = (1+x^3)^2(1 -16x^3 -4x^6 +4x^9 +x^{12}) \,.
\end{equation}
The root $-1$ has multiplicity $2$. The noninteger real roots are approximately
$-1.584586$, $-1.272728$, $0.394958$, and $1.255444$.
They are either less than $-1$ or else positive, but none in the interval $(-1,0)$.

\item
For all tested values $k\ge2$ and $n\in\mathbb{Z}$,
numerical calculations indicate that for a nonvanishing polynomial $\mathcal{F}_{n,k}(x)$,
only the roots $x=0$ and $x^k = -1$, if they exist, ever exhibit a multiplicity greater than one.
All the other roots are simple.
\end{enumerate}
\end{remark}

Recall from eq.~\eqref{eq:structpoly} that all the nonzero roots of a nonvanishing polynomial have the form $x^k = \textrm{(number)}$.
For $n<0$, recall the enumerated list following eq.~\eqref{eq:genfcn_Fibpoly_neg_w_pow}
and the statement \#5 that $\mathcal{F}_{-2k,k} = -x(x^k+2)$.
A more detailed analysis yields the following pattern for $n=-2k$, $n=-3k$, and $n=-4k$:
\begin{equation}  
\begin{split}
  \mathcal{F}_{-2k,k}(x) &= -x(x^k+2) \,,
  \\
  \mathcal{F}_{-3k,k}(x) &= -x(x^k+1)(x^k+3) \,,
  \\
  \mathcal{F}_{-4k,k}(x) &= -x(x^k+1)^2(x^k+4) \qquad\qquad (k\ge3) \,.
\end{split}
\end{equation}
Hence for $n<0$, there are roots where $x^k = -2$ or $x^k = -3$, etc.
Unfortunately the pattern does not extend indefinitely.

\begin{conjecture}\label{conj:xks}
Numerical calculations and symbolic manipulations yield the following.
For $k\ge2$ and $n = -sk$, where $s\in[2,k+1]$ (the pattern fails for $s\ge k+2$),
\begin{equation}  
  \mathcal{F}_{-sk,k}(x) = -x(x^k+1)^{s-2}(x^k+s) \,.
\end{equation}
Hence $\mathcal{F}_{-sk,k}(x)$ has roots $x^k=-s$.
The case $s=2$ was derived in the enumerated list following eq.~\eqref{eq:genfcn_Fibpoly_neg_w_pow}.
These are not the only $(n,k)$ tuples with a root such that $x^k=-j$, where $j>1$ is an integer.
Numerical calculations indicate that for fixed $k\ge2$, there is a finite set of $(n,k)$ tuples with roots $x^k=-j$, for $j\in[2,k+1]$.
All such roots appear to be simple.
For odd $k$, the concomitant real root $-(j^{1/k})$ is negative and less than $-1$.
\end{conjecture}

Next let us study the upper bounds on the amplitudes of the nonzero roots of the polynomial $\mathcal{F}_{n,k}(x)$,
i.e.~the roots of $P_{n,k}(x^k)$.
For $k\ge2$ and $n\in\mathbb{Z}$,
let $\zeta_{n,k}$ denote the maximum amplitude of $|x_{\rm root}|^k$,
where $P_{n,k}(x_{\rm root}^k)=0$.
Set $\zeta_{n,k}=0$ if $\mathcal{F}_{n,k}(x)$ has no nonzero roots, including vanishing polynomials.
Sample results for $\zeta_{n,k}$ are displayed in Figure \ref{fig:upbound}.
For $n>0$ (upper panel), the curves are plotted for $k=2$, $k=3$, and $k=8$.
For $n<0$ (lower panel), we fix $k=4$ and index the curves by the remainder $r_{n,k}$.
The results are surprising.
\begin{enumerate}
\item
  For $n>0$, for fixed $k\ge2$ the value of $\zeta_{n,k}$ approaches a supremum as $n$ increases.
  The value of the supremum is small and decreases as $k$ increases.
  For $k=2$, Hoggatt and Bicknell \cite{HoggattBicknell_roots} proved the supremum is $4$.
  See also Kaymak \cite{Kaymak} for bounds for the roots of the Tribonacci polynomials for $n>0$.
  See the plot of the roots of $\mathcal{F}_{40,8}(x)$ in Figure \ref{fig:roots_poly}:
  the evidence suggests that as $k$ and $n$ increase, the roots spread out in angle but not radius.
  
\item
  For $n<0$, for fixed $k\ge3$ the upper bound $\zeta_{n,k}$ is much larger and increases approximately linearly for large $|n|$.
  The values in Figure \ref{fig:upbound} for $k=4$ and $n<0$ form four branches, indexed by the remainder $r_{n,k}$.
  Similar results were obtained for other values of $k$ (there are $k$ branches for fixed $k$).
  Recall Conjecture \ref{conj:xks}, that for $n<0$, $\mathcal{F}_{-sk,k}(x)$ has roots given by $x^k = -s$, $s\in[1,k+1]$.
  Hence $\zeta_{n,k}$ grows linearly with $|n|$ for $n<0$, i.e.~without bound.
\end{enumerate}
\begin{conjecture}\label{conj:limsup_neg}
Numerical evidence indicates that for even $k\ge2$ and $n>0$, 
the supremum $\limsup_{n\to\infty}\zeta_{n,k}$, say $\hat{\zeta}_k$, is given by the unique real root $\hat{\zeta}_k>1$ of the equation
\begin{equation}
\label{eq:hat_zeta_k}  
\frac{\hat{\zeta}_k}{(\hat{\zeta}_k-1)^{k+1}} = \frac{k^k}{(k+1)^{k+1}} \,.
\end{equation}
For $k=2$, it is easily verified that $\hat{\zeta}_2=4$.
The numerically computed values of $\hat{\zeta}_k$ are tabulated in Table \ref{tb:hat_zeta_k}, for even $k\in[2,20]$.
It is not known why eq.~\eqref{eq:hat_zeta_k} works.
A similar formula for odd $k$ has not been found.
\end{conjecture}
\begin{conjecture}\label{conj:upbound_neg}
For fixed $k\ge3$ and $n<0$, the upper bound is $\zeta_{n,k} \le \lfloor |n|/k\rfloor$.
The bound is attained whenever $r_{n,k}=1$ (equivalently $n=-sk$, where $s\ge1$), i.e.~it is a tight bound.
The slope of the branches in the value of $\zeta_{n,k}$ decreases monotonically as $r_{n,k}$ increases from $1$ through $k-1$.
The lowest branch is indexed by $r_{n,k}=k-1$.
The branch indexed by $r_{n,k}=0$ does not fit a simple pattern.
\end{conjecture}

\setcounter{equation}{0}
\section{Elementary symmetric polynomials of nonzero roots}\label{sec:elem_symm_poly}
All the relevant expressions below have already been derived.
We begin with $n>0$.
Define $N_{n,k} = \lfloor d_{n,k}/k\rfloor = \lfloor(k-1)(n-1)/k\rfloor$ for brevity.
Recall eq.~\eqref{eq:structpoly}:
let us factor out $x^{r_{n,k}}$ and define $\xi=x^k$ and write $P_{n,k}(\xi)$ below.
Let $\xi_{n,k,j}$, $j\in[1,N_{n,k}]$ denote the roots of $P_{n,k}(\xi)$, counting multiplicity,
and denote them collectively by $\vec{\xi}_{n,k}$.
Next denote the elementary symmetric polynomials of the roots of $P_{n,k}(\xi)$ by $\sigma_h(\vec{\xi}_{n,k})$, $h\in[1,N_{n,k}]$.
Since the coefficient of the highest term in $P_{n,k}(\xi)$ is unity, the expression below follows from eqs.~\eqref{eq:poly_pos} and \eqref{eq:C_pos}:
\begin{equation}
\sigma_h(\vec{\xi}_{n,k,1}) = (-1)^h C_k(n-h-1,h) \qquad\qquad (1 \le h \le N_{n,k})\,.
\end{equation}
Recall $q_{n,k}$ and $r_{n,k}$ from eq.~\eqref{eq:q_r_def}.
We omit the subscripts and write simply $q$ and $r$ to avoid cluttering the expressions below.
From eqs.~\eqref{eq:Cnk1_pos} and \eqref{eq:low_pos}, the sum and product of the roots are respectively 
\begin{align*}
  \sum_{j=1}^{N_{n,k}} \xi_{n,k,j} &= -(n-2) \,,
  \\
  \prod_{j=1}^{N_{n,k}} \xi_{n,k,j} &= (-1)^{N_{n,k}} \binom{q+r}{r} \,.
\end{align*}
See \cite{HeSimonRicci_zeros,Kaymak_Ozgur} for results for special cases, e.g.~the Tribonacci polynomials.

As always, for $n<0$ we restrict the analysis to nonvanishing polynomials.
Again define $N_{n,k} = \lfloor d_{n,k}/k\rfloor$, where now $d_{n,k}$ is given by eq.~\eqref{eq:deg_neg}.
For $n<0$, we have seen from Propositions \ref{prop:deg} and \ref{prop:sec_highest_neg}
that the count of the $h^{th}$ highest power of $x$ depends on the remainder $r$.
If $r=0$, then the count is $j_k=h$ but for $r\in[1,k-1]$, the count is $j_k=h+r-1$, where $h\in[0,N_{n,k}]$.
We again denote the elementary symmetric polynomials of the roots of $P_{n,k}(\xi)$ by $\sigma_h(\vec{\xi}_{n,k})$, $h\in[1,N_{n,k}]$.
Then we obtain the following from eqs.~\eqref{eq:poly_neg} and \eqref{eq:high_term_neg}, for $h\in[1,N_{n,k}]$:
\begin{equation}
\sigma_h(\vec{\xi}_{n,k}) = \begin{cases} \displaystyle (-1)^h C_k(-(h+1),|n|+1-k(h+1)) &\quad (r=0) \,,\phantom{\Biggl|} \\
  \displaystyle (-1)^{h+r} \frac{C_k(-(h+r),|n|+1-k(h+r))}{\displaystyle\binom{q-1}{r-1}} &\quad (r\in[1,k-1]) \,.
  \end{cases}
\end{equation}
From eqs.~\eqref{eq:high_term_neg} and \eqref{eq:sec_term_neg}, the sum of the roots is
\begin{equation}
\sum_{j=1}^{N_{n,k}} \xi_{n,k,j} = \begin{cases} 
  -(2q-3) & \qquad (r=0,\, k=2)\,,\\ 
  -(q-1) & \qquad (r=0,\, k\ge3)\,,\phantom{\biggl|} \\
  \displaystyle -\frac{(r+1)(q-r)}{r} &\qquad (r\in[1,k])  \,.
\end{cases}
\end{equation}
From eqs.~\eqref{eq:low_neg} and \eqref{eq:high_term_neg}, the product of the roots is 
\begin{equation}
\label{eq:prod_neg}  
\prod_{j=1}^{N_{n,k}} \xi_{n,k,j} 
  = \begin{cases} (-1)^{N_{n,k}}  &\qquad (r=0) \,,\phantom{\biggl|} \\
    \displaystyle (-1)^{N_{n,k}}\frac{q}{r} &\qquad (r\in[1,k]) \,.
\end{cases}
\end{equation}

\newpage
\section{Errata and updates}\label{sec:errata_updates}
\subsection{Errata}
Zhang \cite{Zhang_privcomm} kindly contacted me, with corrections of some statements in the text above.
\begin{enumerate}
\item
It was stated in Conjecture \ref{conj:upbound_neg}
that for fixed $k\ge3$ and $n<0$, the upper bound is $\zeta_{n,k} \le \lfloor |n|/k\rfloor$.
It was then stated that
the bound is attained whenever $r_{n,k}=1$ (equivalently $n=-sk$, where $s\ge1$), i.e.~it is a tight bound.
\emph{Ma and Zhang \cite{Zhang_privcomm} proved that for $r_{n,k}=1$, then $\zeta_{n,k} = \lfloor |n|/k\rfloor$ for $s\in[1,k+1]$,
but $\zeta_{n,k} < \lfloor |n|/k\rfloor$ (strict inequality) for $s=k+2$.}
Hence the bound is tight in the sense that it is attained for some values of $s$, but it is not attained for all $s\ge1$.
\item
Proposition \ref{prop:poly_extra_struct} is correct, but Ma and Zhang \cite{Zhang_privcomm} showed that my \emph{proof} of it has a loophole.
The following statement, in the proof of Proposition \ref{prop:poly_extra_struct}, is erroneous:
\begin{quote}
It also follows that if $\mathcal{F}_{n,k}$ contains a factor $(x^k+1)$ with multiplicity $m_{n,k}$, then
$\mathcal{F}_{n+k+1,k}$ also contains a factor $(x^k+1)$ with the same multiplicity $m_{n,k}$.
\end{quote}
Ma and Zhang \cite{Zhang_privcomm} have derived a correct proof of Proposition \ref{prop:poly_extra_struct}.
\end{enumerate}

\subsection{Vanishing polynomial coefficients II}
It was stated in Sec.~\ref{sec:van_coeff} that for $k\ge3$ and $n\le0$ there are polynomials $\mathcal{F}_{n,k}(x)$ which are not identically zero but some,
though not all, of the coefficients are zero.
Some examples were given, e.g., the Tribonacci polynomial $T_{-14}(x)$ and the Pentanacci polynomial $P_{-66}(x)$.
It was stated in Sec.~\ref{sec:van_coeff} that 
such ``vanishing coefficients'' have been found for all tested odd values of $k$ (there are no examples for even $k$).

We present here a proof that there are no vanishing coefficients for even $k$.
For $n\le0$, recall eq.~\eqref{eq:poly_neg}, reproduced below:
\begin{equation}
\label{eq:poly_neg1}
\mathcal{F}_{n,k}(x) = \sum_{h=0}^{\lfloor(|n|+1-k)/k\rfloor} C_k(-(1+h), |n|+1-k(1+h))\,x^{|n|+1-k(h+1)} \,.
\end{equation}
The polynomial coefficient $C_k(-(1+h), |n|+1-k(1+h))$ is given by eq.~\eqref{eq:Cnk_neg_qr}, reproduced below:
\begin{equation}
\label{eq:Cnk_neg_qr1}    
  C_k(-(1+h),|n|+1-k(1+h)) = \sum_{s=0}^{q-h-1} (-1)^{k(q-s-h-1)+r} \binom{1+h}{k(q-h-s-1)+r} \binom{h+s}{s} \,.
\end{equation}
\begin{proposition}\label{prop:no_van_coeff_even_k}
For even $k\ge4$ and $n\le0$, if $\mathcal{F}_{n,k}(x)$ is a nonvanishing polynomial, then $\mathcal{F}_{n,k}(x)$ has no vanishing coefficients.
\end{proposition}
\begin{proof}
If $k$ is even, then $(-1)^{k(q-s-h-1)}=1$ and the sum in eq.~\eqref{eq:Cnk_neg_qr1} simplifies to
\begin{equation}
\label{eq:Cnk_neg_qr2}    
  C_k(-(1+h),|n|+1-k(1+h)) = (-1)^r \sum_{s=0}^{q-h-1} \binom{1+h}{k(q-h-s-1)+r} \binom{h+s}{s} \,.
\end{equation}
All the terms in the sum are positive, hence $C_k(-(1+h),|n|+1-k(1+h))\ne0$.
\end{proof}
It was also stated in Sec.~\ref{sec:van_coeff} that there appears to be no simple pattern to determine the vanishing coefficients.
We present here one example of a vanishing coefficient, for odd $k$.
\begin{proposition}\label{prop:van_coeff_odd_k}
For odd $k\ge3$ and $n\le0$, if $\mathcal{F}_{n,k}(x)$ is a nonvanishing polynomial,
then $\mathcal{F}_{n,k}(x)$ has a vanishing coefficient
for $|n|=k^2+2k-1$, i.e., $q_{n,k}=k+2$ and $r_{n,k}=0$.
The coefficient of $x^k$ vanishes, which is indexed by $h=k$ in eq.~\eqref{eq:poly_neg1}.
\end{proposition}
\begin{proof}
Substitute $|n|=k^2+2k-1$, $q=k+2$, $r=0$ and $h=k$ in \eqref{eq:Cnk_neg_qr1} to obtain
\begin{equation}
\begin{split}
\label{eq:Cnk_neg_qr3}    
C_k(-(1+h),|n|+1-k(1+h)) &= \sum_{s=0}^1 (-1)^{k(1-s)} \binom{1+k}{k(1-s)} \binom{k+s}{s}
\\
&= (-1)^k \binom{1+k}{k} \binom{k}{0} + \binom{1+k}{0} \binom{k+1}{1}
\\
&= -(1+k) + (k+1)
\\
&= 0 \,.
\end{split}
\end{equation}
This proves the result.
\end{proof}
\begin{remark}
Proposition \ref{prop:van_coeff_odd_k} does not give a complete listing of all the vanishing coefficients, but it is one example.
Numerical calculations indicate that for odd $k\ge3$ and $n\le0$, then $|n|=k^2+k-2$ is the smallest value of $|n|$ for which a vanishing coefficient appears,
but this claim is not (yet) proved.
\end{remark}

\subsection{Absence of real roots for even $k$}
It was proved in Proposition \ref{prop:realroots} that for even $k\ge2$ and $n\ge3$, the polynomial $\mathcal{F}_{n,k}(x)$ has no nonzero real roots.
For $n<0$, the following was stated in Remark \ref{rem:real_roots}.
\begin{quote}
  For $n<0$ and even $k$, numerical calculations and symbolic manipulations indicate that the coefficients in a nonvanishing polynomial have no sign changes.
  \dots \;
  The reason for this is not known.  
  This has the consequence that, for even $k$ and all $n\in\mathbb{Z}$, the polynomial $\mathcal{F}_{n,k}(x)$ has no nonzero real roots.
\end{quote}
Here we prove that for $n<0$ and even $k\ge2$, the coefficients in a nonvanishing polynomial have no sign changes, hence the polynomial has no nonzero real roots.
\begin{proposition}\label{prop:evenk_nneg_no_real_roots}
  For even $k\ge2$ and $n<0$, the polynomial $\mathcal{F}_{n,k}(x)$ has no nonzero real roots (treating only nonvanishing and nonconstant polynomials).
\end{proposition}
\begin{proof}
The proof relies on eq.~\eqref{eq:Cnk_neg_qr}, specialized to even $k$.
It was proved in eq.~\eqref{eq:Cnk_neg_qr2} that
the coefficient $C_k(-(1+h),|n|+1-k(1+h))$ equals $(-1)^r$ multiplied by a positive number.
For fixed $k$ and $n$, then the value of $r$ is also fixed, hence from eq.~\eqref{eq:poly_neg} (and using $|n|+1=kq+r$),
\begin{equation}
\mathcal{F}_{n,k}(x) = (-1)^rx^r \sum_{h=0}^{\lfloor(|n|+1-k)/k\rfloor} \textrm{(positive)} \times x^{k(q-h-1)} \,.
\end{equation}
All the coefficients in $\mathcal{F}_{n,k}(x)$ have the same sign, hence $\mathcal{F}_{n,k}(x)$ has no positive real roots.
Since $k$ is even, $(-x)^k=x^k$, hence $\mathcal{F}_{n,k}(x)$ also has no negative real roots.
Hence for even $k\ge2$ and all $n\in\mathbb{Z}$, the nonvanishing and nonconstant polynomials $\mathcal{F}_{n,k}(x)$ have no nonzero real roots. 
\end{proof}
\begin{remark}
The author thanks Zhang \cite{Zhang_privcomm} for pointing out that eq.~\eqref{eq:Cnk_neg_qr} furnishes a simple proof of Proposition \ref{prop:evenk_nneg_no_real_roots}.
\end{remark}

\subsection{Bounds on amplitudes of nonzero roots for $n>0$}
Recall from eq.~\eqref{eq:structpoly} that $\mathcal{F}_{n,k}(x) = x^{r_{n,k}} P_{n,k}(x^k)$ and $P_{n,k}(0)\ne0$.
Hence the nonzero roots of $\mathcal{F}_{n,k}(x)$ are the roots of $P_{n,k}(x^k)$.
See Sec.~\ref{sec:poly_roots} for relevant definitions.
For $k\ge2$ and $n\in\mathbb{Z}$,
let $\zeta_{n,k}$ denote the maximum amplitude of $|x_{\rm root}|^k$,
where $P_{n,k}(x_{\rm root}^k)=0$.
Set $\zeta_{n,k}=0$ if $\mathcal{F}_{n,k}(x)$ has no nonzero roots, including vanishing polynomials.

For $n>0$, it was reported in Sec.~\ref{sec:poly_roots} that for fixed $k\ge2$
the value of $\zeta_{n,k}$ is bounded and $\limsup_{n\to\infty}\zeta_{n,k}<\infty$.
See Figure \ref{fig:upbound} (upper panel).
For fixed $k$ and even $n$, Conjecture \ref{conj:limsup_neg}
stated that the value of $\limsup_{n\to\infty}\zeta_{n,k}$ is given by the unique real root of eq.~\eqref{eq:hat_zeta_k}.
\begin{remark}
Zhang \cite{Zhang_privcomm} proved eq.~\eqref{eq:hat_zeta_k} and moreover informed the author of the following.
For all $k\ge2$, let $z$ be a root of the equation
\begin{equation}
\label{eq:hat_zeta_k1}  
\frac{z}{(z-1)^{k+1}} = (-1)^k\frac{k^k}{(k+1)^{k+1}} \,.
\end{equation}
Then $\limsup_{n\to\infty}\zeta_{n,k}=\max |z|$, the largest amplitude of the roots of eq.~\eqref{eq:hat_zeta_k1}.
If $k$ is even, the root with the largest amplitude is real and positive and we obtain eq.~\eqref{eq:hat_zeta_k}, thence the values listed in Table \ref{tb:hat_zeta_k}.
If $k$ is odd, the root with the largest amplitude is one member of a complex conjugate pair.
For $k=2$, Hoggatt and Bicknell \cite{HoggattBicknell_roots} proved that $\limsup_{n\to\infty}\zeta_{n,k}=4$.
For $k=3$, Zhang \cite{Zhang_privcomm} showed that eq.~\eqref{eq:hat_zeta_k1} yields $(3z+1)^2(3z^2-14z+27)=0$.
The roots with the largest amplitude are the roots of the quadratic $3z^2-14z+27$.
The roots are complex and their product is $27/3=9$, hence $\limsup_{n\to\infty}\zeta_{n,3}=\sqrt{9}=3$.
The values of $\hat{\zeta}_k$ are tabulated in Table \ref{tb:hat_zeta_k_all}, for $k\in[2,30]$,
obtained by solving eq.~\eqref{eq:hat_zeta_k1}.
The data are an extension of Table \ref{tb:hat_zeta_k}, to include odd $k$.
Numerical calculations are consistent with the values tabulated in Table \ref{tb:hat_zeta_k_all}.
\end{remark}

\subsection{Bounds on amplitudes of nonzero roots for $n<0$}
It was reported in Sec.~\ref{sec:poly_roots} that 
for $n<0$ and fixed $k\ge3$, the upper bound $\zeta_{n,k}$ increases approximately linearly for large $|n|$.
The values in Figure \ref{fig:upbound} (lower panel) for $k=4$ and $n<0$ form four branches, indexed by the remainder $r_{n,k}$.
Similar results were obtained for other values of $k$ (there are $k$ branches for fixed $k$).
Let us present more detailed results for $\zeta_{n,k}$ for $n<0$.
We update Conjecture \ref{conj:xks} as follows.
\begin{proposition}\label{conj:xks1}
For $n<0$, for fixed $k\ge3$ and $r\in[0,k-1]$ and $q\in[1+r,k+r]$,
\begin{equation}
\label{eq:xks1}  
\mathcal{F}_{1-kq-r,k}(x) = \begin{cases} (x^k+1)^{q-1} \,, & \qquad (r=0) \,, \\ \displaystyle
(-1)^r \binom{q}{r}\,x^r(x^k+1)^{q-r-1}\Bigl(\frac{rx^k}{q} + 1\Bigr) \,, & \qquad (r\in[1,k-1]) \,. \end{cases}
\end{equation}
For $q=k+r+1$, then
\begin{equation}
\label{eq:xks2}  
\mathcal{F}_{1-kq-r,k}(x) = \begin{cases} (x^k+1)^{q-1} +(-1)^kx^k \,, & \qquad (r=0) \,, \\ \displaystyle
(-1)^rx^r \biggl[ \binom{q}{r}(x^k+1)^{q-r-1}\Bigl(\frac{rx^k}{q} + 1\Bigr) +(-1)^kx^k\biggr] \,, & \qquad (r\in[1,k-1]) \,. \end{cases}
\end{equation}
\end{proposition}
\begin{proof}  
I obtained eqs.~\eqref{eq:xks1} and \eqref{eq:xks2} via numerical calculations and symbolic manipulations.
Ma and Zhang \cite{Zhang_privcomm} have derived eqs.~\eqref{eq:xks1} and \eqref{eq:xks2} rigorously.
\end{proof}  
Ma and Zhang \cite{Zhang_privcomm} have proved the following.
\begin{enumerate}
\item
For $r=0$, $\zeta_{n,k}=1$ for $2 \le q \le k$ and $\zeta_{n,k}>1$ for all $q\ge k+1$.
\item
For $r\in[1,k-1]$, $\zeta_{n,k}=q/r$ for $1+r \le q \le k+r$ and $\zeta_{n,k}<q/r$ (strict inequality) for $q=k+r+1$.
\end{enumerate}
\begin{conjecture}
The following remains an open question. For fixed $k\ge3$ and $n<0$ and $r\in[1,k-1]$,
\begin{equation}
\zeta_{n,k} < \frac{q}{r} \qquad (q > k+r+1) \,.
\end{equation}
The behavior of the $r=0$ branch is more complicated, for $q \ge k+1$.
\end{conjecture}

\section*{Acknowledgement}
The author thanks W.~Zhang for helpful communications to improve the manuscript.

\bibliographystyle{amsplain}

\begin{thebibliography}{10}
\bibitem{Arolkar}
S.~Arolkar,
On the derivatives of B-Tribonacci polynomials,
\textit{Notes on Number Theory and Discrete Mathematics}, \textbf{28} (2022), 491--499.  

\bibitem{Arslan_Uslu}
B.~Arslan and K.~Uslu,
A polynomial sequence generalizing an integer sequence associated with Tribonacci numbers,
\textit{European Journal of Science and Technology}, \textbf{36} (2022), 185--190.

\bibitem{GarciaGomezLuca2020}
J.~Garc{\'i}a, C.~A.~G{\'o}mez and F.~Luca,
On the zero-multiplicity of the $k$-generalized Fibonacci sequence,
\textit{Journal of Difference Equations and Applications}, \textbf{26} (2020), 1564--1578.

\bibitem{GarciaGomezLuca2024}
J.~Garc{\'i}a, C.~A.~G{\'o}mez and F.~Luca,
Solving Skolem's problem for the $k$-generalized Fibonacci sequence with negative indices,
\textit{Journal of Number Theory}, \textbf{257} (2024), 273--299.

\bibitem{Gupta_etal}
Y.~K.~Gupta, V.~H.~Badshah, M.~Singh, and K.~Sisodiya,
Some identities of Tribonacci polynomials,
\textit{Turkish Journal of Analysis and Number Theory}, \textbf{4} (2016), 20--22.  

\bibitem{HeSimonRicci_num}
M.~X.~He, D.~Simon, and P.~E.~Ricci,
Numerical results on the zeros of the generalized Fibonacci polynomials,
\textit{Calcolo}, \textbf{34} (1997), 25--40.

\bibitem{HeSimonRicci_zeros}
M.~X.~He, D.~Simon, and P.~E.~Ricci,
Dynamics of the zeros of Fibonacci polynomials, 
\textit{Fibonacci Quart.}, \textbf{35} (1997), 160--168.

\bibitem{HoggattBicknell_roots}
V.~E.~Hoggatt, Jr.~and M.~Bicknell,
Roots of Fibonacci polynomials,
\textit{Fibonacci Quart.}, \textbf{11} (1973), 271--274.

\bibitem{HoggattBicknell_FibGen}
V.~E.~Hoggatt, Jr.~and M.~Bicknell,
Generalized Fibonacci polynomials,
\textit{Fibonacci Quart.}, \textbf{11} (1973), 457--465.

\bibitem{Kaymak}
O.~O.~Kaymak,
Some remarks on the zeros of Tribonacci polynomials,
\textit{International Journal of Analysis and Applications}, \textbf{16} (2018), 368--373.

\bibitem{Kaymak_Ozgur}
O.~O.~Kaymak and N.~{\"O}zg{\"u}r,
On the zeros of $R$-Bonacci polynomials and their derivatives,
\textit{Communications Faculty of Sciences University of Ankara Series A1: Mathematics and Statistics}, \textbf{71} (2022), 978--992.

\bibitem{Koshy}
T.~Koshy, 
\textit{Fibonacci and Lucas Numbers with Applications}, 
Wiley, New York (2001).

\bibitem{Kuhapatanakul_Sukruan}
K.~Kuhapatanakul and L.~Sukruan,
The generalized Tribonacci numbers with negative subscripts,
\textit{Integers}, \textbf{14} (2014), A32.

\bibitem{Ricci}
P.~E.~Ricci,
A note on $\mathcal{Q}$-matrices and higher order Fibonacci polynomials,
\textit{Notes on Number Theory and Discrete Mathematics}, \textbf{27} (2021), 91--100.

\bibitem{Skolem1933}
T.~Skolem,
Einige s{\"a}tze {\"u}ber gewisse Reihenentwicklungen und exponentiale Beziehungen mit Anwendung auf diophantische Gleichungen,
\textit{Oslo Vid.~Akad.~Skrifter}, \textbf{I} (1933) 6 (in German).

\bibitem{Soykan}
Y.~Soykan,
Generalized Tribonacci polynomials,
\textit{Earthline Journal of Mathematical Sciences}, \textbf{13} (2023), 1--120.
  
\bibitem{Tasyurdu_Polat}
Y.~Ta{\c s}yurdu and Y.~E.~Polat,
Tribonacci and Tribonacci-Lucas hybrinomials,
\textit{Journal of Mathematics Research}, \textbf{13} (2021), 32--43.

\bibitem{Waddill}
M.~E.~Waddill,
The Tetranacci sequence and generalizations,
\textit{Fibonacci Quart.}, \textbf{30} (1992), 9--19.

\bibitem{Zhang_privcomm}
W.~Zhang, private communication (2026), and H.~Ma and W.~Zhang, to be submitted.

\end{thebibliography}

\newpage
\begin{table}[!htb]
\centering
{\begin{tabular}{|r|l|l|l|l|lllllllll}
  \hline
$n$ & \multicolumn{1}{c|}{$T_n(x)$} & \multicolumn{1}{c|}{$Q_n(x)$} & \multicolumn{1}{c|}{$P_n(x)$} & \multicolumn{1}{c|}{$d_{n,k}$} \\ \hline
$0$  &&&& \\
$-1$  &&&& \\ 
$-2$  & $1$      &                & & $(0,-,-)$  \\
$-3$  & $-x$     & $1$            & & $(1,0,-)$  \\
$-4$  &          & $-x$           & $1$ & $(-,1,0)$ \\
$-5$  & $1 +x^3$ &                & $-x$ & $(3,-,1)$ \\
$-6$  & $-2x -x^4$  &                & & $(4,-,-)$    \\
$-7$  & $x^2$  & $1+x^4$        & & $(2,4,-)$    \\
$-8$  & $1 +2x^3 +x^6$  & $-2x -x^5$     & & $(6,5,-)$ \\
$-9$  & $-3x -4x^4 -x^7$  & $x^2$          & $1+x^5$ & $(7,2,5)$ \\
$-10$ & $3x^2 +2x^5$  &                & $-2x-x^6$ & $(5,-,6)$ \\
$-11$ & $\dots$  & $1 +2x^4 +x^8$ & $x^2$ & $(9,8,2)$ \\
$-12$ & $\dots$ & $-3x -4x^5 -x^9$       & & $(10,9,-)$ \\
$-13$ & $\dots$ & $3x^2 +2x^6$ & & $(8,6,-)$ \\
$-14$ & $\dots$ & $-x^3$ & $1 +2x^5 +x^{10}$  & $(12,3,10)$ \\
$-15$ & $\dots$ & $\dots$        & $-3x -4x^6 -x^{11}$ & $(13,12,11)$ \\
$-16$ & $\dots$ & $\dots$        & $3x^2 +2x^7$ & $(11,13,7)$ \\
$-17$ & $\dots$ & $\dots$        & $-x^3$ & $(15,10,3)$ \\
$-18$ & $\dots$ & $\dots$        & & $(16,7,-)$ \\
$-19$ & $\dots$ & $\dots$        & $1 +3x^5 +3x^{10} +x^{15}$ & $(14,16,15)$ \\
$-20$ & $\dots$ & $\dots$        & $-4x -9x^6 -6x^{11} -x^{16}$ & $(18,17,16)$ \\
$-21$ & $\dots$ & $\dots$        & $6x^2 +9x^7 +3x^{12}$ & $(19,14,12)$ \\
$-22$ & $\dots$ & $\dots$        & $-4x^3 -3x^8$ & $(17,11,8)$ \\
$-23$ & $\dots$ & $\dots$        & $x^4$ & $(21,20,4)$  \\ \hline
\end{tabular}}
\caption{\label{tb:Fib_poly_neg} Tabulation of the Tribonacci, Quadranacci, and Pentanacci polynomials ($T_n(x)$, $Q_n(x)$, and $P_n(x)$, respectively) for indices $n\le0$.
  Blanks denote the polynomial is identically zero.
  Ellipses ``$\dots$'' denote the polynomial does not vanish, but was not displayed, for clarity of the exposition.
  The degree $d_{n,k}$ of the nonvanishing polynomials is also tabulated.}
\end{table}

\newpage
\begin{table}[!htb]
\centering
{\begin{tabular}{|c|lllllllllllll|}
\hline
$m \setminus x^j$ & $1$ & $x$ & $x^2$ & $x^3$ & $x^4$ & $x^5$ & $x^6$ & $x^7$ & $x^8$ & $x^9$ & $x^{10}$ & $x^{11}$ & $x^{12}$ \\ \hline
$-3$ & $1^i$  & $-3^j$  & $3^k$  & $-1^\ell$  & $3^{\;}$  & $-9^{\;}$  & $9^{\;}$  & $-3^{\;}$  & $6^{\;}$  & $-18^{\;}$  & $18^{\;}$  & $-6^{\;}$ & $\dots$ \\ 
$-2$ & $1^e$  & $-2^f$  & $1^g$  & $\phantom{-}0^h$  & $2^i$  & $-4^j$  & $2^k$  & $\phantom{-}0^\ell$  & $3^{\;}$  & $-6^{\;}$  & $3^{\;}$  & $\phantom{-}0^{\;}$ & $\dots$ \\
$-1$ & $1^a$  & $-1^b$  & $0^c$  & $\phantom{-}0^d$  & $1^e$  & $-1^f$  & $0^g$  & $\phantom{-}0^h$  & $1^i$  & $-1^j$  & $0^k$  & $\phantom{-}0^\ell$ & $\dots$ \\
$0$ & $1^A$ & $0^X$ & $0^Y$ & $0^Z$ & $0^a$ & $0^b$ & $0^c$ & $\phantom{-}0^d$ & $0^e$ & $\phantom{-}0^f$ & $0^g$ & $\phantom{-}0^h$ & $\dots$ \\
$1$ & $1^E$ & $1^D$ & $1^C$ & $1^B$ &&&&&&&&& \\
$2$ & $1$ & $2$ & $3$ & $4$ & $3^E$ & $2^D$ & $1^C$ &&&&&& \\
$3$ & $1$ & $3$ & $6$ & $10$ & $12$ & $12$ & $10$ & $6$ & $3^E$ & $1^D$ &&& \\
$4$ & $1$ & $4$ & $10$ & $20$ & $31$ & $40$ & $44$ & $40$ & $31$ & $20$ & $10$ & $4$ & $1^E$ \\ \hline	
\end{tabular}}
\caption{\label{tb:tbl_k4}
Tabulation of the coefficients in the series expansion of $(1+x+x^2+x^3)^m$ for $-3\le m \le 4$.
A blank denotes the element is zero.
Ellipses ``$\dots$'' indicate the row elements are non-terminating.
Numbers tagged with the same superscript are coefficients in the same Quadranacci polynomial.}
\end{table}

\newpage
\begin{table}[!htb]
\centering
{\begin{tabular}{|r|l|c|l|c|}
  \hline
  $n$ & \multicolumn{1}{c|}{$T_n(x)$} & $\rho_{n,3}$ & \multicolumn{1}{c|}{$Q_n(x)$} & $\rho_{n,4}$ \\ \hline
  $1$ & $1$ & $0$ & $1$ & $0$ \\
  $2$ & $x^2$ & $0$ & $x^3$ & $0$ \\
  $3$ & $x(1+x^3)$ & $1$ & $x^2(1+x^4)$ & $1$ \\
  $4$ & $(1+x^3)^2$ & $2$ & $x(1+x^4)^2$ & $2$ \\
  $5$ & $x^2(3 +3x^3 +x^6)$ & $0$ & $(1+x^4)^3$ & $3$ \\
  $6$ & $x(2 +6x^3 +4x^6 +x^9)$ & $0$ & $x^3(4 +6x^4 +4x^8 +x^{12})$ & $0$ \\
  $7$ & $(1+x^3)(1 + 6x^3 +4x^6 +x^9)$ & $1$ & $x^2(3 +10x^4 +10x^8 +5x^{12} +x^{16})$ & $0$ \\
  $8$ & $x^2 (1+x^3)^2 (6 +4x^3 +x^6)$ & $2$ & $x (1 + x^4) (2 + 10 x^4 + 10 x^8 + 5 x^{12} + x^{16})$ & $1$ \\
  $9$ & && $(1 + x^4)^2 (1 + 10 x^4 + 10 x^8 + 5 x^{12} + x^{16})$ & $2$ \\
  $10$ & && $x^3 (1 + x^4)^3 (10 + 10 x^4 + 5 x^8 + x^{12})$ & $3$ \\ \hline
\end{tabular}}
\caption{\label{tb:Tri_Quad_poly}
  Tabulation of the Tribonacci and Quadranacci polynomials $T_n(x)$ and $Q_n(x)$ for
  $n\in[1,8]$ and $n\in[1,10]$, respectively.
  The corresponding values of $\rho_{n,3}$ and $\rho_{n,4}$ are also tabulated.}
\end{table}

\newpage
\begin{table}[!htb]
  \centering
{\begin{tabular}{|r|l|c|l|c|}
  \hline
  \multicolumn{1}{|c|}{$n$} & \multicolumn{1}{c|}{$T_n(x)$} & $\rho_{n,3}$ & \multicolumn{1}{c|}{$Q_n(x)$} & $\rho_{n,4}$ \\ \hline
  $-10$ & $3x^2 +2x^5$ & $0$ & &  \\
  $-11$ & $1 +2x^3 +3x^6 +x^9$ & $0$ & $(1+x^4)^2$ & $2$ \\
  $-12$ & $-x(1+x^3)^2(4+x^3)$ & $2$ & $-x(1+x^4)(3+x^4)$ & $1$ \\
  $-13$ & $3x^2(1+x^3)(2+x^3)$ & $1$ & $x^2(3+2x^4)$ & $0$ \\
  $-14$ & $1 +3x^6 +4x^9 +x^{12}$ & $0$ & $-x^3$ & $0$ \\
  $-15$ & $-x(5 +15x^3 +18x^6 +8x^9 +x^{12})$ & $0$ & $(1+x^4)^3$ & $3$ \\
  $-16$ & $2x^2(1+x^3)^2(5+2x^3)$ & $2$ & $-x(1+x^4)^2(4+x^4)$ & $2$ \\
  $-17$ & $(1+x^3)(1-x^3) (1-5x^3+5x^6+x^9)$ & $1$ & $3x^2(1+x^4)(2+x^4)$ & $1$ \\
  $-18$ & $-x(6+20x^3+36x^6+30x^9-10x^{12}-x^{15})$ & $0$ & $-x^3(4+3x^4)$ & $0$ \\
  $-19$ & $x^2(15 +49x^3 +60x^6 +30x^9 +5x^{12})$ & $0$ & $1 +5x^4 +6x^8 +4x^{12} +x^{16}$ & $0$ \\ \hline
\end{tabular}}
\caption{\label{tb:Tri_Quad_poly_neg}
  Tabulation of the Tribonacci and Quadranacci polynomials $T_n(x)$ and $Q_n(x)$ for
  $-10\ge n\ge -19$.
  The polynomial $Q_{-10}(x)$ vanishes identically.
  The corresponding values of $\rho_{n,3}$ and $\rho_{n,4}$ are also tabulated.}
\end{table}

\newpage
\begin{table}[!htb]
  \centering
\begin{tabular}{|r|l|l|l|l|l|}
\hline  
$k$ & $\hat{\zeta}_k$ \\ \hline
$2$ & $4$ \\
$4$ & $3.06329$ \\
$6$ & $2.74043$ \\
$8$ & $2.57444$ \\
$10$ & $2.47245$ \\
$12$ & $2.40302$ \\
$14$ & $2.35251$ \\
$16$ & $2.31399$ \\
$18$ & $2.28358$ \\
$20$ & $2.25891$ \\ \hline
\end{tabular}
\caption{\label{tb:hat_zeta_k}
  Numerical estimate for $\hat{\zeta}_k$, the supremum of the maximum amplitude of the $k^{th}$ power of the nonzero roots of the $k$-Fibonacci polynomial $\mathcal{F}_{n,k}(x)$, for even $k$ and $n>0$.
The value $\hat{\zeta}_2=4$ is exact.}
\end{table}

\newpage
\begin{table}[!htb]
  \centering
\begin{tabular}{|r|l|l|l|l|l|}
\hline  
$k$ & $\hat{\zeta}_k$ \\ \hline
$2$ & $4$ \\
$3$ & $3$ \\
$4$ & $3.06329$ \\
$5$ & $2.74568$ \\
$6$ & $2.74043$ \\
$7$ & $2.58521$ \\
$8$ & $2.57444$ \\
$9$ & $2.4823$ \\
$10$ & $2.47245$ \\
$11$ & $2.4113$ \\
$12$ & $2.40302$ \\
$13$ & $2.3594$ \\
$14$ & $2.35251$ \\
$15$ & $2.31977$ \\
$16$ & $2.31399$ \\
$17$ & $2.28849$ \\
$18$ & $2.28358$ \\
$19$ & $2.26312$ \\
$20$ & $2.25891$ \\
$21$ & $2.24213$ \\
$22$ & $2.23847$ \\
$23$ & $2.22444$ \\
$24$ & $2.22124$ \\
$25$ & $2.20933$ \\
$26$ & $2.2065$ \\
$27$ & $2.19625$ \\
$28$ & $2.19373$ \\
$29$ & $2.18482$ \\
$30$ & $2.18256$ \\ \hline
\end{tabular}
\caption{\label{tb:hat_zeta_k_all}
  Value of $\hat{\zeta}_k$, the supremum of the maximum amplitude of the $k^{th}$ power of the nonzero roots of the $k$-Fibonacci polynomial $\mathcal{F}_{n,k}(x)$, for $n>0$.
The values $\hat{\zeta}_2=4$ and $\hat{\zeta}_3=3$ are exact.}
\end{table}

\newpage
\begin{figure}[htb]
\centering
\includegraphics[width=0.99\textwidth]{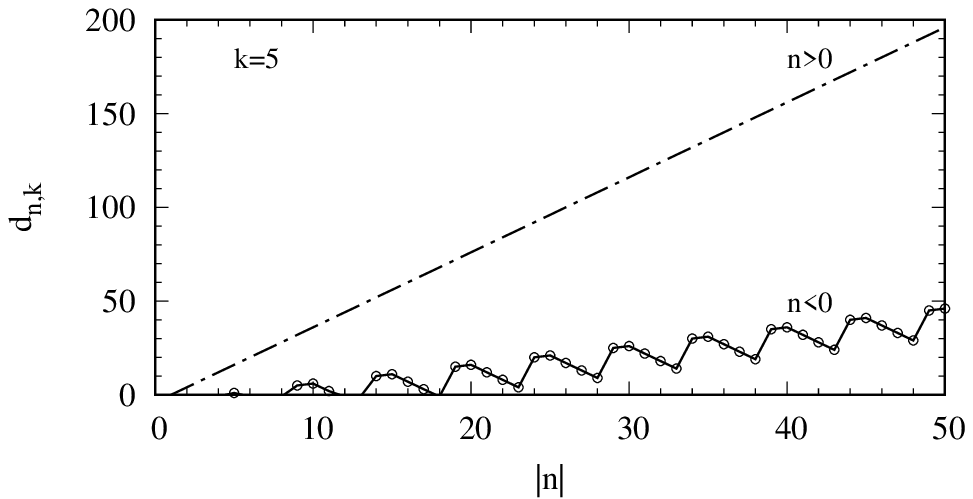}
\caption{\label{fig:deg} Plot of the degree $d_{n,k}$ of a polynomial for $k=5$ and $n>0$ (dotdash line) and $n<0$ (solid line and circles).
  The value of $|n|$ is plotted on the horizontal axis.
  The blanks for $n<0$ occur when a polynomial vanishes identically.}
\end{figure}

\newpage
\begin{figure}[!htb]
\centering
\includegraphics[width=0.99\textwidth]{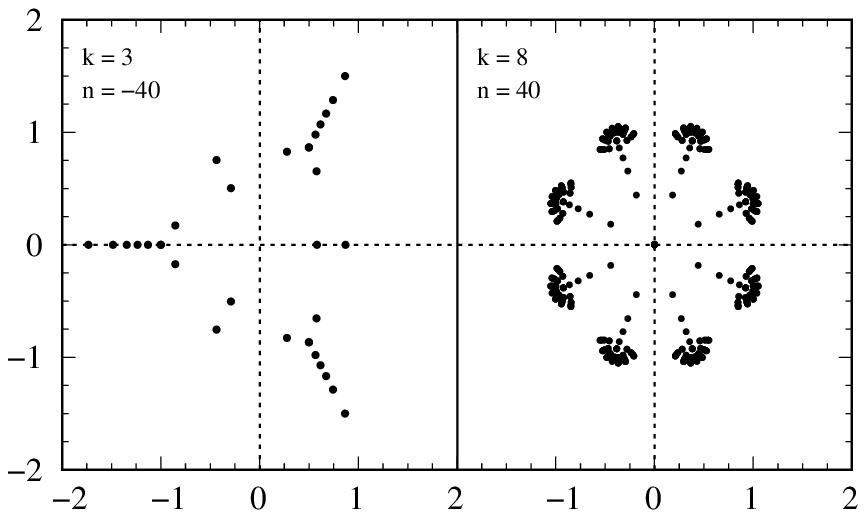}
\caption{\label{fig:roots_poly} Argand diagram plots of the roots of the polynomials
  $\mathcal{F}_{-40,3}(x)$ and
  $\mathcal{F}_{40,8}(x)$.}
\end{figure}

\newpage
\begin{figure}[!htb]
\includegraphics[width=0.99\textwidth]{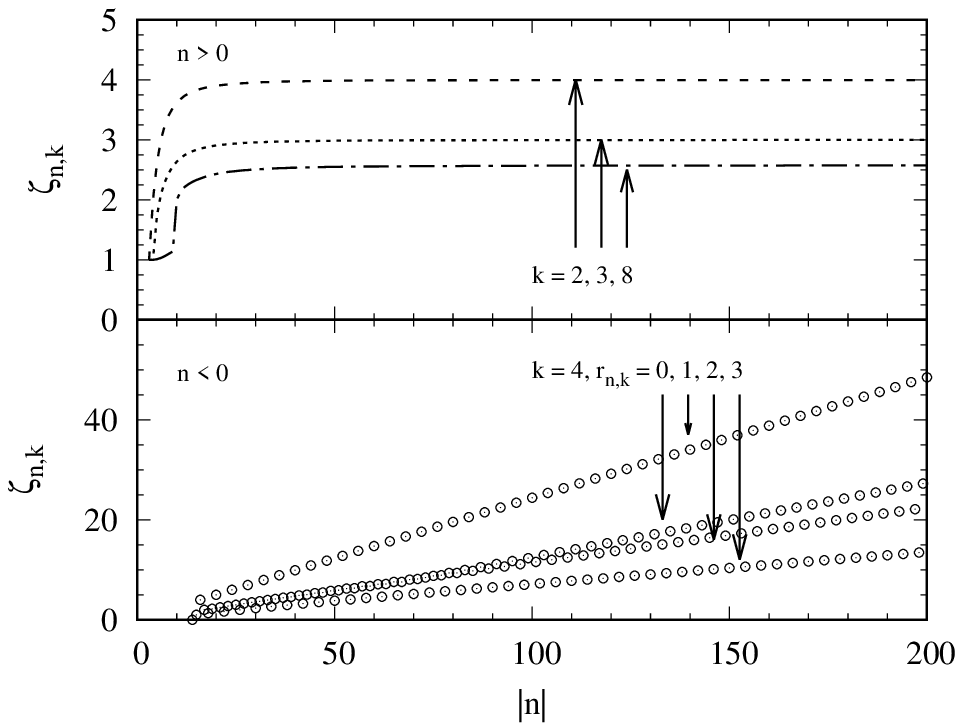}
\caption{\label{fig:upbound} Upper panel: plot of $\zeta_{n,k}$ for $n>0$ and $k=2$, $k=3$, and $k=8$.
  Lower panel: plot of $\zeta_{n,k}$ for $n<0$ and $k=4$, where the curves are indexed by the remainder $r_{n,k}$.}
\end{figure}

\end{document}